\documentclass{amsart}

\newcommand{\be}{\begin{equation}}
\newcommand{\ee}{\end{equation}}
\newcommand{\bea}{\begin{eqnarray}}
\newcommand{\eea}{\end{eqnarray}}
\newcommand{\nn}{\nonumber}
\newcommand{\bee}{\begin{eqnarray*}}
\newcommand{\eee}{\end{eqnarray*}}

\usepackage{cases}
\usepackage{latexsym}
\usepackage{amsmath}
\usepackage[arrow,matrix]{xy}
\usepackage{stmaryrd}
\usepackage{amsfonts}
\usepackage{amsmath,amssymb,amscd,bbm,amsthm,mathrsfs,dsfont}

\usepackage{fancyhdr}
\usepackage{amsxtra,ifthen}
\usepackage{verbatim}

\numberwithin{equation}{section}

\theoremstyle{plain}

\newtheorem*{theorem*}{Theorem}

\newtheorem*{proposition*}{Proposition}

\theoremstyle{definition}

\newtheorem*{example*}{Example}

\newtheorem{fact}{Fact}

 \newtheorem{thm}{Theorem}[section]
 \newtheorem{cor}[thm]{Corollary}
 \newtheorem{lem}[thm]{Lemma}
 
 \newtheorem{prop}[thm]{Proposition}
 \newtheorem{question}[thm]{Question}
 \theoremstyle{definition}
 
\newtheorem{subcase}{\bf Subcase}

\newtheorem{case-x}{\bf Case}
\newtheorem{subcase-x1}{\bf Subcase}
\newtheorem{subcase-x2}{\bf Subcase}
\newtheorem{subcase-x3}{\bf Subcase}
\newtheorem{subcase-x4}{\bf Subcase}
\newtheorem{case-y}{\bf Case}
\newtheorem{subcase-y1}{\bf Subcase}
\newtheorem{subcase-y2}{\bf Subcase}
\newtheorem{subcase-y3}{\bf Subcase}
\newtheorem{subcase-y4}{\bf Subcase}
\newtheorem{subcase-y5}{\bf Subcase}
\newtheorem{subcase-y6}{\bf Subcase}
\newtheorem{subcase-y7}{\bf Subcase}
\newtheorem{subcase-y8}{\bf Subcase}

 \newtheorem{deff}[thm]{Definition}
 
 \theoremstyle{remark}
 \newtheorem{rem}[thm]{Remark}
 \newtheorem{example}[thm]{Example} 
\newtheorem*{theproof-2}{\bf Proof of Proposition \ref{propos}}

\newtheorem*{theproof-1}{\bf Proof of Proposition \ref{propM}}

\usepackage{amsmath}

 \allowdisplaybreaks
 
\begin{document}

\title[Identities involving additive maps on division rings]
{Identities involving additive maps on division rings}
\author{Lovepreet Singh}
\address{Lovepreet Singh,
	Department of Mathematics, Indian Institute of Technology Patna,
	Bihar, 801106, India}
\email{lovepreet.ls0029@gmail.com}

\author{S. K. Tiwari}
\address{S. K. Tiwari,
	Department of Mathematics, Indian Institute of Technology Patna,
	Bihar, 801106, India}
\email{shaileshiitd84@gmail.com \& sktiwari@iitp.ac.in}

 \thanks{ }


\dedicatory{}

\commby{}
\begin{abstract}
Let $g$ be an additive map on a division ring $D$. In this paper, we study the functional identity $G_{1}(y)g(y)G_{2}(y) = H(y)$, where $G_{1}(Y), G_{2}(Y)$, $H(Y)$ are generalized polynomials in $D_{G}[Y]$ such that both $G_{1}(Y)$ and $G_{2}(Y)$ are non-zero. As an application, we prove that for a non-commutative division ring $D$ of $\operatorname{char}(D) \neq 2$, any pair of additive maps $g_1,g_2: D \to D$ satisfying $g_{1}(y)y^{-m} + y^{n}g_{2}(y^{-1})= 0$ is $ g_{1} = g_{2} = 0$, provided that $m$ and $n$ are positive integers with $(m,n) \neq (1,1)$. 
\end{abstract}

\maketitle

\section{Introduction}
Rings and algebras considered in this paper are associative and unital. For any ring $R$, denote by $R^{*}$ its group of invertible elements. A map $g: R \rightarrow R$ is called an additive map if 

\begin{equation}\label{cauchy}
	g(x+y)=g(x)+g(y)
\end{equation} 
for all $x, y \in R$. 

The equation \eqref{cauchy} was first studied by Cauchy  \cite[p.103]{cauchy1821} for continuous real-valued functions of a real variable. Later, this classical equation led to the study of functional identities involving the inverse.

 One of the earliest problems involving inverses was posed by Halperin in $1963$ \cite{h1}. He asked whether every additive map $g:\mathbb{R}\to\mathbb{R}$ satisfying $g(y)=y^2g(y^{-1})$ must necessarily be continuous.
In $1964$, Kurepa \cite{k1} generalized this problem considerably. He proved that, if  $g_{1}, g_{2}$ are two non-zero additive maps and $P$ is a continuous function on $\mathbb{R}$ with $P(1)=1$ such that $g_{1}(y) = P(y)g_{2}(y^{-1})$ for all $y \in \mathbb{R}^{\ast}$, then $P(y) = y^{2}$, $g_{1}(y) + g_{2}(y) = 2yg_{2}(1)$, and the map $y \mapsto g_{1}(y) - y g_{1}(1)$ is a derivation.  Further in $1968$, Nishiyama and Horinouchi \cite{n1} characterized the maps satisfying $g(y^{m_{1}}) = ay^{m_{2}}g(y^{m_{3}})$ for all $y \in \mathbb{R}$, where $m_{1},m_{2}, \ \text{and} \ m_{3}$ are the integers. Subsequently, Kannappan and Kurepa \cite{pl1,pl2} investigated several functional identities involving inverses over $\mathbb{R}$ during $1970-1971$, leading to further developments in the area.

Recall that an additive map $\delta$ on a ring $R$ is a derivation if $\delta(y_{1}y_{2}) = \delta(y_{1})y_{2} + y_{1} \delta (y_{2})$ for all $y_1, y_2 \in R$. Differentiating $yy^{-1} =1$ gives $\delta (y) =-y \delta (y^{-1})y$ for $y \in R^{*}$. Conversely one
asks which additive maps are forced to be of this form. In $1995$, Bre\v{s}ar \cite{bres1} showed that additive $g_{1},g_{2}$ on division ring $D$ with $g_{1}(y)y = yg_{2}(y)$ for all $y$ are given by $g_{1}(y) = ya + \eta (y) $ and $g_{2}(y) = ay + \eta (y)$ for fixed $a \in D$ and additive central $\eta$. In $2018$, Catalano \cite{lc1} set a way for extending rational identities to functional identities and proved the following result.

\begin{thm}(\cite[Theorem $1$]{lc1}) \label{catmain}
	Let $D$ be a non-commutative division ring of $\operatorname{char}(D) \neq 2$. If $g_1$ and $g_2$ are additive maps  on $D$ such that
	\begin{equation} \label{ceq1}
		g_{1}(y)y^{-1} + yg_{2}(y^{-1})=0
	\end{equation}
	for all $y \in D^{*}$. Then there exists $a \in D$ such that $g_{1}(y) = ya + \delta (y)$ and $g_{2}(y) = -ay + \delta (y)$ for all $y \in D$, where $\delta$ is a derivation on $D$.
\end{thm}

 Catalano \cite{lc1} proved the matrix analogue of Theorem \ref{catmain} and asked whether the restriction $\operatorname{char} D \neq 3$ imposed in \cite[Theorem $4$]{lc1} is needed. Latter, Arga\c c et. al. \cite{narg1} extended Catalano's results in $2020$ and gave an affirmative answer by proving the following theorem.

\begin{thm}  (\cite[Theorem $1.5$]{narg1}) \label{catmain2}
	Let $D$ be a division ring of $\operatorname{char}(D) \neq 2$ and let $R=M_{t}(D)$, where $t>1$. Suppose  $g_{1}$ and $g_{2}$ are additive maps on $R$ satisfying $	g_{1}(y)y^{-1}  + yg_{2}(y^{-1})=0$
	for all $y \in R^{*}$. Then  there exists $a \in R$ such that $g_{1}(y) =ya + \delta (y) $ and $g_{2}(y) = - ay +\delta (y) $ for all $y \in R$, where $\delta$ is a derivation.
\end{thm}




The identity $g_1(y)=-y^n g_2(y^{-1})$ was first studied over fields by Ng \cite{ng1}. Catalano and Merch\'an \cite{lc2} extended the problem to division rings, while Lee and Lin \cite{tk1} subsequently completed the division ring case.

\begin{thm} (\cite[Theorem $5.1$]{tk1}) \label{leemain}
 Let $D$ be a non-commutative division ring of $\operatorname{char}(D)\neq 2$, and let $n$ be a positive integer different from $2$.
 If $g_{1}$ and $g_{2}$ are additive maps on $D$ such that
	\begin{equation} \label{leeeq}
		g_{1}(y) = y^{n}g_{2}(y^{-1})
	\end{equation}
	for all $y \in D^{*}$, then $g_{1}=g_{2}=0$.
\end{thm}
For $n=2$, the equation \eqref{leeeq} is solved by Dar and Jing \cite{na1}. They obtained $g_{1}(y) = ya $ and $g_{2}(y) = ya$ for all $y \in D$, where $a \in D$. The assumption non-commutative or $\operatorname{char}(D) \neq 2$ can not be removed in Theorem \ref{leemain}. For example, if we take $D= \mathbb{Z}_{2}$ and $g_{1}(y) = g_{2}(y) = y$, then $g_{1}$ and $g_{2}$ satisfy the equation \eqref{leeeq}, but both $g_1$ and $g_2$ are non-zero.

Motivated by the above results, it is natural to ask the characterization of additive maps $g_{1},g_{2}: D \rightarrow D$ satisfying the identity $g_{1}(y)y^{-m} + y^{n}g_{2}(y^{-1}) =0,$ where $n, m$ are positive integers. Our approach to the above question is based on the approach of \cite{tk1}, arising from generalized polynomials with coefficients in a division ring $D$. 
Denote by $D_G[Y]$ the generalized free algebra over $Z(D)$ generated by $Y$. Its elements are finite $Z(D)$-linear combinations of generalized monomials. For a monomial of degree $s$, we use the
representation
\begin{equation*}
	q_{1}Yq_{2}Yq_{3} \cdots q_{s}Yq_{s+1},
\end{equation*}
where $q_{i} \in D$. For variables  $Y_1, \cdots, Y_m$, the generalized free algebra over $Z(D)$  is represented by $D_{G}[Y_1, \cdots, Y_m ]$. For more details of generalized algebras see \cite{GoMo}.

\begin{deff}
	A map $g: D\rightarrow D$ is said to be an elementary operator if there exist finitely many $p_{i},q_{i} \in D$ such that $g(y) = \sum_{i} p_{i}yq_{i}$ for all $y \in D$.
\end{deff}

\begin{thm} \label{mainB}
	Let $g: D \to D$ to be an additive map and let  $G_1(Y), G_2(Y),\\ H(Y) \in D_{G}[Y]$ such that $G_1(y)g(y)G_2(y) = H(y)$ for all $y \in D$. If both $G_1(Y) $ and $G_{2}(Y)$ are non-zero, then either $g$ is an elementary operator or $D$ is finite-dimensional over its center $Z(D)$.
\end{thm}

\begin{thm}\label{main2}
	Let $D$ be a non-commutative division ring of $\operatorname{char}(D) \neq 2$. If $g_{1}$ and $g_{2}$ are additive maps on $D$ such that
	\begin{equation}  \label{maineq1}
		g_{1}(y)y^{-m} + y^ng_{2}(y^{-1})=0
	\end{equation}
	for all $y \in D^ *$, where $m$ and $n$ are positive integers with $(m,n) \not = (1,1)$, then $g_{1} = g_{2} = 0$.
\end{thm}

We observe that for  $n = 1$ and $m \geq2$, if additive maps $g_{1},g_{2}$ on $D$ satisfy the identity $g_{1}(y)y^{-m} + yg_{2}(y^{-1}) = 0$ for all $y \in D^{*}$, then by Theorem \ref{main2} we get $g_{1} = g_{2}= 0$. As an application of this result, we have the following corollary.

\begin{cor} \label{dcor}
	Let $D$ be a non-commutative division ring with $\operatorname{char}(D) \neq 2$ and $ a \in D^{*}, \ell \in D$. If $g_{1}$ and $g_{2}$ are additive maps on $D$ such that $g_{1}(y_{1}) y_{2}^{m} + y_{1} g_{2}(y_{2}) = \ell$
	for all $y_{1}, y_{2} \in D$ with $y_{1}y_{2}=a$, where  $a\in Z(D)$, $m\geq 2$ is a positive integer, then $g_{1}(y) = 0$ and  $g_{2}(y)= ya^{-1}\ell$ for all $y \in D$.
\end{cor}

\begin{proof}
	Let $y_{1} \in D^{*}$ and $y_{2} = y_{1}^{-1}a$. It follows that $g_{1}(y_{1})y_{1}^{-m}a^{m} + y_{1}g_{2}(y_{1}^{-1}a) = \ell$. This implies
	\begin{equation} \label{eqc2}
		g_{1}(y_{1})y_{1}^{-m}a^{m} + y_{1}(g_{2}(y_{1}^{-1}a)-y_{1}^{-1}\ell) = 0.
	\end{equation}
	Multiplying equation \eqref{eqc2} by $a^{-m}$ from the right hand side, we get
	\begin{equation}
		g_{1}(y_{1})y_{1}^{-m} + y_{1}(g_{2}(y_{1}^{-1}a)-y_{1}^{-1}\ell)a^{-m} = 0.
	\end{equation}
It follows from Theorem \ref{main2} that $g_{1}(y_{1})= 0$ and $g_{2}(y_{1}) = y_{1}a^{-1}\ell$ for all $y_{1} \in D$.
\end{proof}

The assumption non-commutative with $\operatorname{char}(D) \neq 2$ is essential in Corollary \ref{dcor}. Following examples validate our fact. 
 
\begin{example}
		 Let $R = \mathbb{Z}_{2}$ be a ring and define maps $g_{1}, g_{2}$ on $R$ such that $g_{1}(y) = -y^{2}$ and $g_{2}(y) = y + y^{2}$ for all $y \in R$. It is clear that both $g_1$ and $g_2$ are additive maps on $R$. For $m=3$, $a = 1 = \ell$, clearly $g_1$ and $g_2$ are satisfying the identity $g_{1}(y)y^{3} + yg_{2}(y^{-1}) = 1$ for all $y \in R^{*}$, but the forms of $g_1$ and $g_2$ are not same as mentioned in Corollary \ref{dcor}.
\end{example}

\begin{example}
	 Let $R = M_{t}(\mathbb{Z}_{4})$  be a non-commutative ring, where $t \geq 2$. Define $g_{1},g_{2}: R \rightarrow R$ as $g_{1}(y) = g_{2}(y) = y$ for all $y \in R$. It is easy to verify that $g_1$ and $g_2$ are additive maps. For $m=1$, $a = 1$ and $ \ell = 2$, it is clear that $g_{1}(y_{1})y_{2} + y_{1}g(y_{2}) = 2$ for all $y_{1},y_{2} \in R^{*}$ with $y_{1}y_{2}=1=y_{2}y_{1}$. However, $g_{1}$ and $g_{2}$ are not in the  same form as mentioned in Corollary \ref{dcor}.
\end{example}
Here is an outline of the paper.
In Section $2$, we provide a proof of Theorem \ref{mainB} and explore its implications and possible generalizations. In Section $3$, we investigate the identity $g(y^{2}) = w_{1}(y)g(y) + g(y)w_{2}(y) + w_{3}(y)g(y)w_{4}(y)$, where $w_{1},w_{2},w_{3},w_{4}$ are arbitrary maps on $D$ and $g$ is an additive map on $D$. In Section $4$, we apply these characterizations and their implications to solve the equation \eqref{maineq1}. In particular, we provide a proof of Theorem \ref{main2}.

\section{The identity $G_1(y)g(y)G_2(y) = H(y)$}

Throughout, $D$ is division ring with center $Z(D)$. We call $D$ a division GPI-ring if some nonzero 
$g(Y_1,\ldots,Y_m)\in D_G[Y_1,\ldots,Y_m]$ vanishes identically on $D$, such a $g$ is a nontrivial GPI of $D$.

A special case of Martindale's structure theorem yields the following result.
\begin{thm}[\cite{M1}, Theorem 3]\label{m1}
	If $D$ is a division ring satisfying a generalized polynomial identity, then $
	[D:Z(D)]<\infty.$
\end{thm}

%
%
%
%
%
%
%
%


We shall repeatedly linearize expressions of the form
$$
G(y)=\sum_i p_{i1}\,g_{i1}(y)\,p_{i2}\dots p_{is}\,g_{is}(y)\,p_{i\,s+1},
$$
where $p_{ij}\in D,\ g_{ij}\colon D\to D$ additive maps.
Put $G^{(1)}(y_1)=G(y_1)$ and, for $k\ge 1$,
$$
G^{(k+1)}(y_1,\dots,y_{k+1})
=G^{(k)}(y_1,\dots,y_{k-1},y_k+y_{k+1})
-G^{(k)}(y_1,\dots,y_k)
-G^{(k)}(y_1,\dots,y_{k-1},y_{k+1}).
$$
An easy induction yields
$$
G^{(s)}(y_1,\dots,y_s)
=\sum_i\sum_{\sigma\in S_s}
p_{i1}\,g_{i1}(y_{\sigma(1)})\,p_{i2}\dots
p_{is}\,g_{is}(y_{\sigma(s)})\,p_{i\,s+1},
$$
 $S_s$ is the symmetric group on $\{1,\dots,s\}$, while $G^{(k)}\equiv 0$
for every $k>s$. A hat over a slot indicates its omission, that is
$$G^{(k)}(y_1,\dots,\widehat{y_j},\dots,y_k)
:=G^{(k)}(y_1,\dots,y_{j-1},y_{j+1},\dots,y_k).$$
When $G=\sum_s G_s$ splits into homogeneous terms $G_s$, we set
$G^{(t)}=\sum_s G_s^{(t)}$. We recall the following observation from \cite[Remark 2.2]{tk1}.

\begin{rem} \label{r2}
	Let $0\neq G(Y)\in D_G[Y]$.
	
	\begin{itemize}
		\item[(i)]
		If $\deg G(Y)>1$, then $
		G(Y_1+Y_2)-G(Y_1)-G(Y_2)\neq 0$ in $D_G[Y_1,Y_2]$.
		
		\item[(ii)]
		Assume that $G(0)=0$. Then $G(Y_1+Y_2)-G(Y_1)-G(Y_2)\neq 0$ if and only if
		$ \deg G(Y)>1.$
		
		\item[(iii)]
		The equality $
		G(Y_1+Y_2)=G(Y_1)+G(Y_2)$
		holds if and only if there exist finitely many $p_i,q_i\in D$ such that $G(Y)=\sum_i p_iYq_i.$
		
		\item[(iv)]
		If $s=\deg G(Y)\geq 1$ and $G(0)=0$, then $G^{(s)}(Y_1,\ldots,Y_s)$
		is a non-zero multilinear generalized polynomial in $D_G[Y_1,\ldots,Y_s]$.
	\end{itemize}
\end{rem}
%
%
%

Now we are in the position to prove the first main theorem. To prove Theorem \ref{mainB}, we need the following result.
\begin{fact} \label{t3}
	\cite[Theorem $2.3$]{tk1} Let $g: D \to D$ to be an additive map  and let $G(Y), H(Y) \in D_{G}[Y]$ such that $G(y)g(y) = H(y)$ for all $y \in D$. If $G(Y)$ is a non-zero, then either $g$ admits a representation as an elementary operator or $D$ is finite-dimensional over its center $Z(D)$.
\end{fact}

\begin{proof}[\textbf{Proof of Theorem \ref{mainB}}]
The argument adapts \cite[proof of Theorem 2.3]{tk1} to the two-sided setting. Assume that $\operatorname{deg} G_1(Y) = s_1$ and $\operatorname{deg} G_2(Y) = s_2$. First suppose $s_2=0$. Then  $G_2(Y) = b^{'}$ for some $b^{'} \in D^{*}$ and hence
	\begin{equation*}
		G_1(y)g(y)b^{'} = H(y).
	\end{equation*}
	Define an additive map $g^{'}: D \to D$ by  $g^{'}(y) = g(y)b^{'}$. 
	Using Fact \ref{t3}, we obtain the desired conclusion. By symmetry, the same argument applies when $s_1 = 0$.  
	
Suppose instead $s_1 = s_2 = 0$, so 
$G_1(Y) = b$ and $G_2(Y) = b^{'}$ for fixed $b,b^{'} \in D^{*}$.
	Then, $g(y) = b^{-1} H(y) b^{'-1}$, and additivity of $g$ forces
	$H(y_{1}+y_{2}) = H(y_{1}) + H(y_{2})$, $H(0) = 0$.
	
	If $\operatorname{deg} H(Y) \leq 1$, then by Remark \ref{r2} (iii), $H(Y) = \sum_i p_i Y q_i$ for suitable finitely many $p_i, q_i \in D$, which implies $g$ is an elementary operator.
 
    We now consider the case where $\operatorname{deg} H(Y) > 1$. By Remark \ref{r2} (i),
	\begin{equation} \label{ts1}
		H(Y_{1} + Y_{2}) - H(Y_{1}) - H(Y_{2}) \neq 0
	\end{equation}
	in $D_{G}[Y_{1},Y_{2}]$. On the other hand, the additivity of $g$ implies that this generalized polynomial vanishes on $D$. Hence $D$ satisfies a non-trivial GPI-algebra. By Theorem \ref{m1}, $[D: Z(D)] < \infty$.
	
	Finally, let $\operatorname{deg} G_1(Y) = s_1 \geq 1$ and $\operatorname{deg} G_2(Y) = s_2 \geq 1$. Write
	\begin{equation*}
		G_{i}(Y) = G_{i0}(Y) + G_{i1}(Y),
	\end{equation*}
	where $G_{i1}(Y)$ is the homogeneous part of $G_{i}(Y)$ of degree $s_i$, $i \in \{1,2\}$.
	Let 
	\begin{center}
		$A_{\ell k}(y) = G_{1\ell}(y)g(y)G_{2k}(y)$ \hspace{0.3cm} and \hspace{0.3cm} $A_{11}(y) = G_{11}(y) g(Y)G_{21}(y)$
	\end{center}
	for all $y \in D$, where $(\ell,k) \neq (1,1)$ and $\ell,k \in  \{0,1\}$. Since $\operatorname{deg} G_{10}(Y) < s_1$ and $\operatorname{deg} G_{20}(Y) < s_2 $, we have
	\begin{equation*}
		A_{\ell k}^{(s_1+s_2+1)}(y_1, \cdots, y_{s_1 + s_2 + 1})= 0
	\end{equation*}
	and 
	\begin{align*}
		A_{11}^{(s_1+s_2+1)} & (y_1, \cdots, y_{s_1 + s_2 + 1})= \\              & \left( \sum_{j=1}^{s_1 + s_2+ 1} G^{(s_1)}_{11}(y_1,\cdots, \Hat{y}_{j},\Hat{y}_{j+1}, \cdots, \Hat{y}_{j+s_2}, \cdots, y_{s_1+s_2+1})g(y_j) \right)\\
		& \left( G_{21}^{(s_2)}(y_{j+1}, \cdots, y_{j+s_2}) \right)
	\end{align*}
	for all $y_i \in D$, where $y_{j+i} = y_{j+i-s_{1}-s_{2}-1}$ for $j+i>s_1+s_2+1$, $i \in \{ 1,2, \cdots, s_2 \}$, $(\ell,k) \neq (1,1)$ and $\ell,k \in \{0,1\}$. It is given that $G_1(y)g(y)G_2(y) = H(y) $ for all $y \in D$, we obtain 
	\begin{equation} \label{ts2}
		\begin{aligned}
			H^{(s_1+s_2+1)} & (y_1, \cdots, y_{s_1 + s_2 + 1})= \\              & \left( \sum_{j=1}^{s_1 + s_2+ 1} G^{(s_1)}_{11}(y_1,\cdots, \Hat{y}_{j},\Hat{y}_{j+1}, \cdots, \Hat{y}_{j+s_2}, \cdots, y_{s_1+s_2+1})g(y_j) \right)\\
			& \left( G_{21}^{(s_2)}(y_{j+1}, \cdots, y_{j+s_2}) \right)
		\end{aligned}
	\end{equation}
	for all $y_i \in D$, where $y_{j+i} = y_{j+i-s_{1}-s_{2}-1}$ for $j+i>s_1+s_2+1$ and $i \in \{ 1,2, \cdots, s_2 \}$. 
	By Remark \ref{r2} (iv), we get
	\begin{center}
		$G^{(s_1)}_{11}(Y_1,\cdots, \Hat{Y}_{j},\Hat{Y}_{j+1}, \cdots, \Hat{Y}_{j+s_2}, \cdots, Y_{s_1+s_2+1})$
	\end{center}
	and
	\begin{center}
		$G_{21}^{(s_2)}(Y_{j+1}, \cdots, Y_{j+s_2})$
	\end{center}
	are multilinear and non-zero.
	Thus, we re-write the equation (\ref{ts2}) as
	\begin{equation} \label{ts3}
		\begin{aligned}
			G^{(s_1)}_{11}&(y_{s_2+1},\cdots, y_{s_1+s_2})g(y_{s_1+s_2+1})
			G_{21}^{(s_2)}(y_{1}, \cdots, y_{s_2})=\\
			& \sum_{j=1}^{s_{1}+ s_{2}}b_{j}(y_{1}, \cdots, y_{s_1 +s_2}) y_{s_1 +s_2+1}c_{j}(y_{1}, \cdots, y_{s_1+s_2}) + \\
			&H^{(s_1+s_2+1)} (y_1, \cdots, y_{s_1 + s_2 + 1})
		\end{aligned}
	\end{equation}
	for all $y_i \in D$, where $b_{j},c_{j}$ are generalized monomials in $y_{1}, \cdots, y_{s_1+s_2},g(y_{1}),\\
	 \cdots, g(y_{s_1+s_2})$. Consider the following cases.
	
	\textbf{Case $1$:} If $G^{(s_1)}_{11}(y_{s_2+1},\cdots, y_{s_2+s_1}) =0$ and $G_{21}^{(s_2)}(y_{1}, \cdots, y_{s_2}) = 0$ for all
	$y_{1}, \cdots, y_{s_1+s_2} \in D$, then $D$ is a division GPI-algebra. It follows from Theorem \ref{m1} that $D$ is finite-dimensional over its center.
	
	\textbf{Case $2$:} If $G^{(s_1)}_{11}(z_{s_2+1},\cdots, z_{s_2+s_1}) \neq 0$ for some $z_{s_2+1}, \cdots, z_{s_1+s_2} \in D$ and $G_{21}^{(s_2)}(y_{1}, \cdots, y_{s_2}) = 0$ for all $y_{1}, \cdots, y_{s_2} \in D$, then by using the same argument as in Case $1$, we get our conclusion.
	
	\textbf{Case $3$:} If $G^{(s_1)}_{11}(y_{s_2+1},\cdots, y_{s_2+s_1}) = 0$ for all $y_{s_2+1}, \cdots, y_{s_1+s_2} \in D$ and\\
	 $G_{21}^{(s_2)}(z_{1}, \cdots, z_{s_2}) \neq 0$ for some $z_{1}, \cdots, z_{s_2} \in D$, then by using the same argument as in Case $1$, we get $[D: Z(D)] < \infty$.
	
	\textbf{Case $4$:} Suppose $G^{(s_1)}_{11}(z_{s_2+1},\cdots, z_{s_2+s_1}) \neq 0$  and $G_{21}^{(s_2)}(z_{1}, \cdots, z_{s_2}) \neq 0$ for some $z_{1}, \cdots, z_{s_1+s_2} \in D$. Let
	\begin{center}
		$\Tilde{b}_{j}= b_{j}(z_1, \cdots, z_{s_1+s_2})$ \hspace{0.3cm} and \hspace{0.3cm} $\Tilde{c}_{j}(z_{1}, \cdots, z_{s_1+s_2})$
	\end{center}
	for all $j$. Let $t = G^{(s_1)}_{11}(z_{s_2+1},\cdots, z_{s_2+s_1}) $ and $ t^{'}= G_{21}^{(s_2)}(z_{1}, \cdots, z_{s_2})$. Given equation $(\ref{ts3})$, we obtain
	\begin{equation} \label{ts4}
		g(y_{s_1+s_2})= \sum_{j} t^{-1} \Tilde{b}_{j}y_{s_1+s_2+1} \Tilde{c}_{j}t^{'-1} + E(y_{s_1+s_2})
	\end{equation}
	for all $y_{s_1+s_2+1} \in D$, where
	\begin{equation*}
		E(Y_{s_1+s_2+1}) = t^{-1}H^{(s_1 + s_2 + 1)}(z_{1}, \cdots, z_{s}, Y_{s_1+s_2+1})t^{'-1} \in D_G[ Y_{s_1 + s_2 +1} ].
	\end{equation*}
	If $E(Y_{s_1+s_2+1})$ is linear in $Y_{s_1 +s_2+1}$, then equation (\ref{ts4}) implies that $g$ is an elementary operator. If $ \operatorname{deg}(E(Y_{s_1+s_2+1})) > 1$, then the additivity of $g$ together with equation (\ref{ts4}), implies that $D$ satisfies the GPI
	\begin{equation*}
		E(y_{s_1 +s_2+1} + y_{s_1+s_2+2})-  E(y_{s_1 + s_2+1}) -   E(y_{s+2}).      
	\end{equation*}
	By Remark \ref{r2} (i), $ E(Y_{s_1 + s_2 +1} + Y_{s_1 + s_2 +2})-  E(Y_{s_1 + s_2 +1}) -   E(Y_{s_1 + s_2 +2}) \neq 0 $. By Theorem \ref{m1}, we conclude that $D$ is finite-dimensional over $Z(D)$.
\end{proof}

\begin{thm} \label{tc5}
	Suppose $g_{1}, \cdots, g_{s}$ are additive maps on a division ring $D$ and 
	$G_{1j}(Y_{1}, \cdots, Y_{s})$, $G_{2j}(Y_{1}, \cdots, Y_{s})$, $H(Y_{1}, \cdots, Y_{s}) \in D_{G}[Y_{1}, \cdots, Y_{s}]$, for $j=1, \cdots, s$  such that
	\begin{equation} \label{ts5}
		\sum_{j=1}^{s}G_{1j}(y_{1}, \cdots, y_{s})g_{j}(y_{j})G_{2j}(y_{1}, \cdots,y_{s}) = H(y_{1}, \cdots, y_{s})
	\end{equation}
	for all $y_i \in D$.
	 If $[D: Z(D)] = \infty$ and $G_{1j}(Y_{1}, \cdots, Y_{s}) \neq 0$, $G_{2j}(Y_{1}, \cdots, Y_{s}) \neq 0$ for some $j$, then $g_{j}$ is an elementary operator.
\end{thm}

\begin{proof}
	Assume that $G_{1t}(Y_{1}, \cdots, Y_{s}) \neq 0$ and
	$ G_{2t}(Y_{1}, \cdots, Y_{s}) \neq 0 $ for some $t$. We show that $g_{t}$ must be an elementary operator. Indeed, if either $ G_{1t}(Y_{1}, \cdots, Y_{s}) $ or $ G_{2t}(Y_{1}, \cdots, Y_{s}) $ were a generalized polynomial identity of $D$, then Theorem \ref{m1} would imply that $[D: Z(D)] < \infty$, which is a contradiction to our assumption. Therefore neither $ G_{1t}(Y_{1}, \cdots, Y_{s}) $ nor $ G_{2t}(Y_{1}, \cdots, Y_{s}) $ is a GPI for $D$. Consequently, one can choose elements $p_1, \cdots, p_s \in D$ such that $G_{1t}(p_{1}, \cdots, p_{s}) \neq 0 $ and $ G_{2t}(p_{1}, \cdots, p_{s}) \neq 0 $. By equation (\ref{ts5}), we have
	\begin{equation} \label{ts6}
		\begin{aligned}
			G_{1t}& (p_{1},\cdots,y_{t}, \cdots, p_{s})g_{t}(y_{t})G_{2t}(p_{1},\cdots,y_{t}, \cdots, p_{s})= \\
			& -\sum_{j=1, j\neq t}^{s} G_{1j}(p_{1},\cdots,y_{t}, \cdots, p_{s})g_{j}(p_{j}) G_{2j}(p_{1},\cdots, y_{t}, \cdots, p_{s}) \\
			& + H(p_{1},\cdots, y_{t}, \cdots, p_{s})
		\end{aligned}
	\end{equation}
	for all $y_t \in D$. Note that
	\begin{equation*}
		\begin{aligned}
			G_{1t}& (p_{1},\cdots, Y_{t}, \cdots, p_{s}),G_{2t}(p_{1},\cdots, Y_{t}, \cdots, p_{s}), \\
			& -\sum_{j=1, j \neq t}^{s} G_{1j}(p_{1},\cdots, Y_{t}, \cdots, p_{s})g_{j}(p_{j}) G_{2j}(p_{1},\cdots,Y_{t}, \cdots, p_{s}) \\
			& + H(p_{1},\cdots,Y_{t}, \cdots, p_{s}) \in D_{G}[Y_{t}].
		\end{aligned}
	\end{equation*}
	Since $ G_{1t}(p_{1},\cdots,Y_{t}, \cdots, p_{s}) \neq 0,G_{2t}(p_{1},\cdots, Y_{t}, \cdots, p_{s}) \neq 0 $ and $[D:Z(D)] = \infty$, by Theorem \ref{mainB}, we have $g_{t}$ is an elementary operator.
\end{proof}

\section{The identity $g(y^2) = w_{1}(y)g(y)+ g(y)w_{2}(y) +  w_{3}(y)g(y)w_{4}(y)$}

A particular case of the identity $g(y^2) = w_{1}(y)g(y)+ g(y)w_{2}(y) +  w_{3}(y)g(y)w_{4}(y)$ appears in Section $4$. The identity treated here generalizes $g(y^2)=w_{1}(y)g(y)$, studied in
\cite[Section 3]{tk1}. The main novelty is that the more general identity admits nontrivial solutions in situations where the degrees of the involved polynomials $w_i$ coincide.

\begin{lem} \label{tl6}
Let $g_1$ and $g_2$ be additive maps on a division ring $D$ of $\operatorname{char}(D) \neq 2$. If $w_{1},w_{2},w_{3},w_{4}: D \rightarrow D$ are maps such that $g_{1}(y^2) = w_{1}(y)g_{2}(y)+ g_{2}(y)w_{2}(y) +  w_{3}(y)g_{2}(y)w_{4}(y)$ for all $ y \in D$, then 
	\begin{equation*}
		\begin{aligned}
					&\Bigl(2w_{1}(2y_{1})- 2w_{1}(y_{1}) -w_{1}(y_{1}+y_{2}) -w_{1}(y_{1}-y_{2}) \Bigl)g_{2}(y_{1}) \\
		& + g_{2}(y_{1})\Bigl(2w_{2}(2y_{1}) - 2w_{2}(y_{1}) - w_{2}(y_{1}+y_{2}) - w_{2}(y_{1}-y_{2})\Bigl)\\
		&+ \Bigl(2w_{1}(y_{2}) -w_{1}(y_{1}+y_{2}) + w_{1}(y_{1}-y_{2})\Bigl)g_{2}(y_{2})\\
		& + g_{2}(y_{2}) \Bigl(2w_{2}(y_{2}) -w_{2}(y_{1}+y_{2}) + w_{2}(y_{1}-y_{2})\Bigl)\\
		& + 2w_{3}(2y_{1})g_{2}(y_{1})w_{4}(2y_{1}) - w_{3}(y_{1}+y_{2})g_{2}(y_{1}+y_{2})w_{4}(y_{1}+y_{2})\\
		& - w_{3}(y_{1}-y_{2})g_{2}(y_{1}-y_{2})w_{4}(y_{1}-y_{2}) - 2w_{3}(y_{1})g_{2}(y_{1})w_{4}(y_{1}) \\
		& + 2w_{3}(y_{2})g_{2}(y_{2})w_{4}(y_{2}) = 0
		\end{aligned}
	\end{equation*}
	for all $y_{1},y_{2} \in D$.
\end{lem}

\begin{proof}
	Let $y_{1},y_{2} \in D$. The identity $y_1y_2+y_2y_1=(y_1+y_2)^2-y_1^2-y_2^2$ gives
	\begin{equation} \label{ts8prod}
		g_1(y_1y_2+y_2y_1)=g_1\bigl((y_1+y_2)^2\bigr)-g_1(y_1^2)-g_1(y_2^2).
	\end{equation}
	Expanding each term on the right of \eqref{ts8prod} using the hypothesis
	\begin{equation}
		\begin{aligned} \label{ts8}
			g_{1}(y_{1}y_{2}+y_{2}y_{1})  &= w_{1}(y_{1} + y_{2})g_{2}(y_{1} +y_{2}) + g_{2}(y_{1} + y_{2})w_{2}(y_{1}+y_{2}) \\ 
			&+ w_{3}(y_{1}+y_{2})g_{2}(y_{1}+y_{2})w_{4}(y_{1}+y_{2}) - w_{1}(y_{1})g_{2}(y_{1})\\
			& - g_{2}(y_{1})w_{2}(y_{1}) - w_{1}(y_{2})g_{2}(y_{2}) - g_{2}(y_{2})w_{2}(y_{2})\\
			& - w_{3}(y_{1})g_{2}(y_{1})w_{4}(y_{1})- w_{3}(y_{2})g_{2}(y_{2})w_{4}(y_{2}).
		\end{aligned}
	\end{equation}
Substituting $(y_1+y_2,\,y_1-y_2)$ for $(y_1,y_2)$ in equation (\ref{ts8}) yields
	\begin{equation*}
		\begin{aligned}
			g_{1}\Bigl((y_{1}+y_{2})(y_{1}-y_{2})& + (y_{1}-y_{2})(y_{1}+y_{2})\Bigl)\\
			&=2w_{1}(2y_{1})g_{2}(y_{1}) + 2g_{2}(y_{1})w_{2}(2y_{1})- w_{1}(y_{1}+y_{2})g_{2}(y_{1}+y_{2})\\
			&- g_{2}(y_{1}+y_{2})w_{2}(y_{1}+y_{2})- w_{1}(y_{1}-y_{2})g_{2}(y_{1}-y_{2})\\
			&- g_{2}(y_{1}-y_{2})w_{2}(y_{1}-y_{2})+
			2w_{3}(2y_{1})g_{2}(y_{1})w_{4}(2y_{1}) \\
			&- w_{3}(y_{1}+y_{2})g_{2}(y_{1}+y_{2})w_{4}(y_{1}+y_{2})\\
			&- w_{3}(y_{1}-y_{2})g_{2}(y_{1}-y_{2})w_{4}(y_{1}-y_{2})\\
			& = (2w_{1}(2y_{1}) -w_{1}(y_{1}+y_{2}) -w_{1}(y_{1}-y_{2}) )g_{2}(y_{1}) \\
			&  + g_{2}(y_{1})(2w_{2}(2y_{1}) - w_{2}(y_{1}+y_{2}) - w_{2}(y_{1}-y_{2}))\\
			& + (-w_{1}(y_{1}+y_{2}) + w_{1}(y_{1}-y_{2}))g_{2}(y_{2}) \\
			& + g_{2}(y_{2}) (-w_{2}(y_{1}+y_{2}) + w_{2}(y_{1}-y_{2})) +
				2w_{3}(2y_{1})g_{2}(y_{1})w_{4}(2y_{1}) \\
			&- w_{3}(y_{1}+y_{2})g_{2}(y_{1}+y_{2})w_{4}(y_{1}+y_{2})\\
			&- w_{3}(y_{1}-y_{2})g_{2}(y_{1}-y_{2})w_{4}(y_{1}-y_{2}).
		\end{aligned}
	\end{equation*}
	On the other hand, we have
	\begin{equation*}
		\begin{aligned}
			&g_{1}\Bigl((y_{1}+y_{2})(y_{1}-y_{2}) + (y_{1}-y_{2})(y_{1}+y_{2})\Bigl)\\
			& = 2g_{1}(y_{1}^{2}) - 2g_{1}(y_{2}^{2}) \\
			& = 2w_{1}(y_{1})g_{2}(y_{1}) + 2g_{2}(y_{1})w_{2}(y_{1}) - 2w_{1}(y_{2})g_{2}(y_{2}) -2g_{2}(y_{2})w_{2}(y_{2})\\
			 & + 2w_{3}(y_{1})g_{2}(y_{1})w_{4}(y_{1}) - 2w_{3}(y_{2})g_{2}(y_{2})w_{4}(y_{2}).  
		\end{aligned}
	\end{equation*}
	Comparing the above two expressions, we get
	\begin{equation*}
		\begin{aligned}
			&\Bigl(2w_{1}(2y_{1})- 2w_{1}(y_{1}) -w_{1}(y_{1}+y_{2}) -w_{1}(y_{1}-y_{2})\Bigl)g_{2}(y_{1}) \\
			& + g_{2}(y_{1})\Bigl(2w_{2}(2y_{1}) - 2w_{2}(y_{1}) - w_{2}(y_{1}+y_{2}) - w_{2}(y_{1}-y_{2})\Bigl)\\
			&+ \Bigl(2w_{1}(y_{2}) -w_{1}(y_{1}+y_{2}) + w_{1}(y_{1}-y_{2})\Bigl)g_{2}(y_{2})\\
			& + g_{2}(y_{2}) \Bigl(2w_{2}(y_{2}) -w_{2}(y_{1}+y_{2}) + w_{2}(y_{1}-y_{2})\Bigl)\\
			& + 2w_{3}(2y_{1})g_{2}(y_{1})w_{4}(2y_{1}) - w_{3}(y_{1}+y_{2})g_{2}(y_{1}+y_{2})w_{4}(y_{1}+y_{2})\\
			& - w_{3}(y_{1}-y_{2})g_{2}(y_{1}-y_{2})w_{4}(y_{1}-y_{2}) - 2w_{3}(y_{1})g_{2}(y_{1})w_{4}(y_{1}) \\
			& + 2w_{3}(y_{2})g_{2}(y_{2})w_{4}(y_{2}) = 0,
		\end{aligned}
	\end{equation*}
	our conclusion.
\end{proof}


\begin{thm} \label{ttl2}
	Let $g_1$ and $g_2$ be additive maps on a division ring $D$ of $\operatorname{char}(D) \neq 2$, and let $w_{i}(Y) \in D_{G}[Y]$ for $i=1,2,3,4$ satisfy $$g_{1}(y^{2}) =  w_{1}(y)g_{2}(y)+ g_{2}(y)w_{2}(y) +  w_{3}(y)g_{2}(y)w_{4}(y)$$ for all $y \in D$. If $\operatorname{deg} w_{i}(Y) > 1$, then either $D$ is finite-dimensional over its center $Z(D)$ or $g_{2}$ is an elementary operator.
\end{thm}

\begin{proof}
Assume each $w_{i}$ exceeds degree $1$ and $D$ is infinite dimensional over its center. By Remark \ref{r2} (i) each
	$ w_{1}(Y_{1}+Y_{2}) - w_{1}(Y_{1}) - w_{1}(Y_{2}), w_{2}(Y_{1}+Y_{2}) - w_{2}(Y_{1}) - w_{2}(Y_{2}), w_{3}(Y_{1}+Y_{2}) - w_{3}(Y_{1}) - w_{3}(Y_{2}), 	w_{4}(Y_{1}+Y_{2}) - w_{4}(Y_{1}) - w_{4}(Y_{2})$ is a non-zero element of $D_{G}[Y_{1},Y_{2}]$.
	Thus, Lemma \ref{tl6} implies that 
	\begin{equation*}
		\begin{aligned}
			&\Bigl(2w_{1}(2y_{1})- 2w_{1}(y_{1}) -w_{1}(y_{1}+y_{2}) -w_{1}(y_{1}-y_{2})\Bigl)g_{2}(y_{1}) \\
			& + g_{2}(y_{1})\Bigl(2w_{2}(2y_{1}) - 2w_{2}(y_{1}) - w_{2}(y_{1}+y_{2}) - w_{2}(y_{1}-y_{2})\Bigl)\\
			&+ \Bigl(2w_{1}(y_{2}) -w_{1}(y_{1}+y_{2}) + w_{1}(y_{1}-y_{2})\Bigl)g_{2}(y_{2})\\
			& + g_{2}(y_{2}) \Bigl(2w_{2}(y_{2}) -w_{2}(y_{1}+y_{2}) + w_{2}(y_{1}-y_{2})\Bigl)\\
			& + 2w_{3}(2y_{1})g_{2}(y_{1})w_{4}(2y_{1}) - w_{3}(y_{1}+y_{2})g_{2}(y_{1}+y_{2})w_{4}(y_{1}+y_{2})\\
			& - w_{3}(y_{1}-y_{2})g_{2}(y_{1}-y_{2})w_{4}(y_{1}-y_{2}) - 2w_{3}(y_{1})g_{2}(y_{1})w_{4}(y_{1}) \\
			& + 2w_{3}(y_{2})g_{2}(y_{2})w_{4}(y_{2}) = 0.
		\end{aligned}
	\end{equation*}
	If $g_{2}$ is not an elementary operator, then by Theorem \ref{tc5}, we get $w_{3}(Y) = 0$ and $w_{4}(Y) = 0$, which leads to a contradiction. Hence, $g_{2}$ is an elementary operator.  
\end{proof}

Before proving next lemma, we need the following fact.

\begin{fact} \cite[Theorem $2(a)$]{M1} \label{ll3}
	Suppose $\{ p_{1}, \cdots, p_{s}\}$ and $\{ q_{1}, \cdots , q_{s}\}$ are two linearly independent subsets of a division ring $D$ over $Z(D)$. Then $\sum_{j} p_{j} Y q_{j} \neq 0 \in D_{G}[Y]$.
\end{fact}

\begin{lem} \label{tll4}
	 Suppose that $\{p_{1},\ldots,p_{s}\}$ and $\{q_{1},\ldots,q_{s}\}$ are linearly independent subsets of a division ring $D$ over $Z(D)$. Let
	 $w_{1}(Y),w_{2}(Y),w_{3}(Y),w_{4}(Y)\in D_{G}[Y]$, and let $
	 d=\max\{\ell_{1},\ell_{2},\ell_{3}+\ell_{4}\}$,
	 where $\ell_i=\deg w_i(Y)$ for $i=1,2,3,4$. If $d$ is attained by exactly one of the integers $\ell_{1}$, $\ell_{2}$, and $\ell_{3}+\ell_{4}$, then
	 $$
	 \sum_{j=1}^{s} p_{j}Y^{2}q_{j}
	 -\sum_{j=1}^{s}\bigl(w_{1}(Y)p_{j}Yq_{j}
	 +p_{j}Yq_{j}w_{2}(Y)\bigr)
	 -\sum_{j=1}^{s}w_{3}(Y)p_{j}Yq_{j}w_{4}(Y)\neq0.
	 $$
\end{lem}

\begin{proof}
	Assume that $\operatorname{deg} w_{i}(Y) = \ell_{i} > 1$, for $i \in \{1,2,3,4\}$. Write $w_{i}(Y) = \sum_{k_{\ell_{i}}=0}^{\ell_{i}} w_{ik_{\ell_{i}}}(Y)$, where $w_{ik_{\ell_{i}}}$ denotes the homogeneous part of $w_{i}(Y)$ of degree $k_{\ell_{i}}$. If no such part exists for a given $k_{\ell_{i}}$, then $w_{ik_{\ell_{i}}}(Y)=0$. In particular, $w_{i\ell_{i}}(Y) \neq 0$ for $i \in \{1,2,3,4\}$. If possible, assume that
	\begin{equation*}
			\sum_{j=1}^{s} p_{j} Y^{2} q_{j} - \sum_{j=1}^{s}\Bigl( w_{1}(Y) p_{j} Y q_{j} + p_{j} Y q_{j} w_{2}(Y)\Bigl) - \sum_{j=1}^{s} w_{3}(Y) p_{j} Y q_{j} w_{4}(Y) = 0.
	\end{equation*}
	Then
	\begin{equation*}
		\begin{aligned}
				\sum_{j=1}^{s} p_{j} Y^{2} q_{j} - &\sum_{k_{\ell_{1}}=0}^{\ell_{1}} w_{1k_{\ell_{1}}}(Y)(\sum_{j=1}^{s} p_{j} Y q_{j}) - (\sum_{j=1}^{s} p_{j} Y q_{j}) \sum_{k_{\ell_{2}=0}}^{\ell_{2}} w_{2k_{\ell_{2}}}(Y)\\
			& -\sum_{k_{\ell_{3}}=0}^{\ell_{3}} w_{3k_{\ell_{3}}}(Y)(\sum_{j=1}^{s} p_{j} Y q_{j}) \sum_{k_{\ell_{4}=0}}^{\ell_{4}} w_{4k_{\ell_{4}}}(Y) = 0.
		\end{aligned}
	\end{equation*}
	 It follows that $$w_{1\ell_{1}}(Y)(\sum_{j=1}^{s}p_{j}Yq_{j})=0, (\sum_{j=1}^{s}p_{j}Yq_{j})w_{2k\ell_{2}}(Y)=0  \ \text{or} \ w_{3\ell_{3}}(Y)(\sum_{j=1}^{s}p_{j}Yq_{j})w_{4\ell_{4}}(Y)=0$$ according as $d=\ell_1$, $d=\ell_2$ or $d=\ell_3+\ell_4$. Further, by \cite[Corollary, p.$379$]{coh1} and \cite[Theorem $2.5$]{cohn2}, we obtain $\sum_{j=1}^{s} p_{j} Y q_{j} = 0$, contradicting Fact \ref{ll3}.
\end{proof}


\begin{lem} \label{remlem1}
	Suppose $\{ p_{1}, \cdots, p_{s}\}$ and $\{ q_{1}, \cdots , q_{s}\}$ are two linearly independent subsets of a division ring $D$ over $Z(D)$. Let
	$w_{1}(Y),w_{2}(Y),w_{3}(Y),w_{4}(Y)\in D_{G}[Y]$ be homogeneous polynomials of degrees $\ell_{1},\ell_{2},\ell_{3},\ell_{4}>1$,
	respectively.  Then
	\begin{equation*}
		\sum_{j=1}^{s} p_{j} Y^{2} q_{j} - \sum_{j=1}^{s} \Bigl(w_{1}(Y) p_{j} Y q_{j} + p_{j} Y q_{j} w_{2}(Y)\Bigl) - \sum_{j=1}^{s}w_{3}(Y)p_{j}Yq_{j}w_{4}(Y) \neq  0.
	\end{equation*}
\end{lem}
\begin{proof}
	Since $w_i(Y)$ is homogeneous of degree $\ell_i>1$ for $i\in\{1,2,3,4\}$, the terms $w_1(Y)p_jYq_j$, $p_jYq_jw_2(Y)$ and $w_3(Y)p_jYq_jw_4(Y)$ are homogeneous	of degrees $\ell_1+1$, $\ell_2+1$ and $\ell_3+\ell_4+1$, respectively, each exceeding $2$.
	If \begin{equation*}
		\sum_{j=1}^{s} p_{j} Y^{2} q_{j} - \sum_{j=1}^{s} \Bigl(w_{1}(Y) p_{j} Y q_{j} + p_{j} Y q_{j} w_{2}(Y)\Bigl) - \sum_{j=1}^{s}w_{3}(Y)p_{j}Yq_{j}w_{4}(Y)=  0,
	\end{equation*} 
	then it follows that $\sum_{j=1}^{s} p_{j} Y^{2} q_{j}=0$, a contradiction by Fact \ref{ll3}.
\end{proof}

The next three examples show that Lemma \ref{tll4} and Lemma \ref{remlem1} may fail when the maximum
$\max\{\ell_1,\ell_2,\ell_3+\ell_4\}$ is attained more than once and the
$w_i(Y)$ are not homogeneous.

	
	\begin{example}
		Suppose $w_{1}(Y) = Y + Y^{4}  $, $w_{2}(Y) = -2Y^{4}, w_{3}(Y) = w_{4}(Y)= Y^{2}$ and $s=1$ with $p_{1} = q_{1}=1$. Then $Y^{2} - w_{1}(Y)Y - Yw_{2}(Y) - w_{3}(Y)Yw_{4}(Y) = 0$.
	\end{example}
	
	\begin{example}
	   Suppose $w_{1}(Y) = Y + Y^{5} $, $w_{2}(Y) = -Y^{5}-Y^{4}, w_{3}(Y) = w_{4}(Y)= Y^{2}$ and $s=1$ with $p_{1} = q_{1}=1$. Then $Y^{2} - w_{1}(Y)Y - Yw_{2}(Y) - w_{3}(Y)Yw_{4}(Y) = 0$. 
	\end{example}
	 
	 \begin{example}
	 	Suppose $w_{1}(Y) = -Y^{2} - Y^{4} $, $w_{2}(Y) = Y + Y^{2}, w_{3}(Y) = w_{4}(Y)= Y^{2}$ and $s=1$ with $p_{1} = q_{1}=1$. Then it is easy to verify that $Y^{2} - w_{1}(Y)Y - Yw_{2}(Y) - w_{3}(Y)Yw_{4}(Y) = 0$.
	 \end{example}

	 This leaves open the following question.
	   
	  \begin{question}
	  	Suppose $\{ p_{1}, \cdots, p_{s}\}$ and $\{ q_{1}, \cdots , q_{s}\}$ are two linearly independent subsets of a division ring $D$ over $Z(D)$ and $w_{i}(Y)\in D_{G}[Y]$ with $\operatorname{deg} w_{i}(Y) > 1$, where $i\in \{1,2,3,4\}$.  Is it possible to conclude either 
	  		\begin{equation*}
	  		\sum_{j=1}^{s} p_{j} Y^{2} q_{j} - \sum_{j=1}^{s}\Bigl( w_{1}(Y) p_{j} Y q_{j} + p_{j} Y q_{j} w_{2}(Y)\Bigl) - \sum_{j=1}^{s} w_{3}(Y) p_{j} Y q_{j} w_{4}(Y) \neq 0
	  	\end{equation*}
	  	or $D$ is finite-dimensional over its center?
	  \end{question}

\begin{thm} \label{tcl5}
Let $D$ be a division ring with $\operatorname{char}(D)\neq 2$, and let $g:D\to D$ be an additive map. Assume that there exist generalized polynomials $w_{1}(Y), w_{2}(Y),w_{3}(Y),\\ w_{4}(Y) \in D_{G}[Y]$ such that $$g(y^2) = w_{1}(y)g(y) + g(y)w_{2}(y) + w_{3}(y)g(y)w_{4}(y)$$ for all $y \in D$. If $\deg w_i(Y)>1$ for $i\in\{1,2,3,4\}$, and either
$\max\{\deg w_1,\deg w_2,\deg w_3+\deg w_4\}$ is attained uniquely or all $w_i(Y)$ are homogeneous, then $D$ is finite-dimensional over its center $Z(D)$.
\end{thm}

\begin{proof}
	By Theorem \ref{ttl2}, either $D$ is a finite-dimensional over $Z(D)$ or $g$ is an elementary operator. In the latter case, there exist finitely many $p_{j}, q_{j} \in D$ satisfying $\sum_{j=1}^{s} p_{j} y q_{j}$ for all $y \in D$. We may assume $s$ to be minimal. Then $\{p_{1}, \cdots, p_{s} \}$ and $ \{ q_{1}, \cdots, q_{s} \} $ are two independent sets over $Z(D)$. Since $g(y^2) = w_{1}(y)g(y) + g(y)w_{2}(y) + w_{3}(y)g(y)w_{4}(y)$ for all $y \in D$, it follows that 
	\begin{equation*}
		\sum_{j=1}^{s} p_{j} y^{2} q_{j} - \sum_{j=1}^{s}\Bigl( w_{1}(y) p_{j}y q_{j} + p_{j} y q_{j} w_{2}(y)\Bigl) - \sum_{j=1}^{s} w_{3}(y) p_{j} y q_{j} w_{4}(y) = 0
	\end{equation*}
	for all $y \in D$. Recall that $\operatorname{deg} w_{i}(Y) > 1$ for $i\in \{1,2,3,4\}$. From Lemma \ref{tll4} it follows  that
	\begin{equation*}
	\sum_{j=1}^{s} p_{j} Y^{2} q_{j} - \sum_{j=1}^{s}\Bigl( w_{1}(Y) p_{j} Y q_{j} + p_{j} Y q_{j} w_{2}(Y)\Bigl) - \sum_{j=1}^{s} w_{3}(Y) p_{j} Y q_{j} w_{4}(Y)
	\end{equation*}
	is a non-trivial generalized polynomial identity for $D$. By Theorem \ref{m1}, we get $D$ has a finite dimesion over its center.
\end{proof}

\begin{lem} \label{l310}
	Let $D$ be a division ring with $\operatorname{char}(D)\neq 2$, and let $g:D\to D$ be an additive map. Assume that there exist generalized polynomials $w_{1}(Y), w_{2}(Y) \in D_{G}[Y]$ such that $$g(y^2) = w_{1}(y)g(y) + g(y)w_{2}(y)$$ for all $y \in D$. If $\operatorname{deg}(w_{i}(Y)>1$ and  all $w_i(Y)$ are homogeneous, then $D$ is finite-dimensional over its center $Z(D)$.
\end{lem}
\begin{proof}
	Suppose $\operatorname{deg}w_{i}(Y) = \ell_{i}>1$ for $i=1,2$, and $D$ is infinite dimensional over its center $Z(D)$. By Remark \ref{r2} (i)  both $w_{1}(Y_1 + Y_2) - w_{1}(Y_1) - w_{1}(Y_2), w_2(Y_1 + Y_2) - w_2(Y_1) - w_2(Y_2)$ is a non-zero element of $D_{G}[Y_1, Y_2]$.  Thus, Lemma \ref{tl6} implies that
	\begin{equation*}
		\begin{aligned}
			&\Bigl(2w_{1}(2y_{1})- 2w_{1}(y_{1}) -w_{1}(y_{1}+y_{2}) -w_{1}(y_{1}-y_{2}) \Bigl)g(y_{1}) \\
			& + g(y_{1})\Bigl(2w_{2}(2y_{1}) - 2w_{2}(y_{1}) - w_{2}(y_{1}+y_{2}) - w_{2}(y_{1}-y_{2})\Bigl)\\
			&+ \Bigl(2w_{1}(y_{2}) -w_{1}(y_{1}+y_{2}) + w_{1}(y_{1}-y_{2})\Bigl)g(y_{2})\\
			& + g(y_{2}) \Bigl(2w_{2}(y_{2}) -w_{2}(y_{1}+y_{2}) + w_{2}(y_{1}-y_{2})\Bigl)=0
		\end{aligned}
	\end{equation*}
If $g$ is not an elementary operator, then by Theorem \ref{tcl5}, we get 
\begin{equation} \label{l3101}
	2w_{i}(2Y_{1})- 2w_{i}(Y_{1}) -w_{i}(Y_{1}+Y_{2}) -w_{i}(Y_{1}-Y_{2})=0
\end{equation}
and 
\begin{equation} \label{l3102}
	2w_{i}(Y_{2}) -w_{i}(Y_{1}+Y_{2}) + w_{i}(Y_{1}-Y_{2})=0,
\end{equation}
where $i= 1,2$.
Since $w_{i}$ are homogeneous of degree greater than $1$, this implies that $w_{i}(0)=0$. Now, replacing $Y_2$ by $Y_1$ in equation \eqref{l3101}, we obtain $w_{i}(2Y_{1}) = 2w_{i}(Y_{1})$. Since $\operatorname{char}(D) \neq 2$, from equaion \eqref{l3101} and equation \eqref{l3102} we obtain $w_{i}(Y_1 + Y_2) = w_{i}(Y_1) + w_{i}(Y_2)$, a contradiction. Thus, $g$ is an elementary operator and there exist finitely many $p_{j}, q_{j} \in D$ satisfying $\sum_{j=1}^{s} p_{j} y q_{j}$ for all $y \in D$. We may assume $s$ to be minimal. Then $\{p_{1}, \cdots, p_{s} \}$ and $ \{ q_{1}, \cdots, q_{s} \} $ are two independent sets over $Z(D)$. By our hypothesis, we have
\begin{equation*}
	\sum_{j=1}^{s} p_{j} y^{2} q_{j} - \sum_{j=1}^{s}\Bigl( w_{1}(y) p_{j}y q_{j} + p_{j} y q_{j} w_{2}(y)\Bigl) = 0
\end{equation*}
for all $y \in D$.  If possible assume that 
\begin{equation*}
	\sum_{j=1}^{s} p_{j} Y^{2} q_{j} - \sum_{j=1}^{s}\Bigl( w_{1}(Y) p_{j} Y q_{j} + p_{j} Y q_{j} w_{2}(Y)\Bigl) = 0.
\end{equation*}
Then it follows that $\sum_{j=1}^{s} p_{j} Y^{2} q_{j} =0$, a contradiction by Fact \ref{ll3}. Theorefore, $\sum_{j=1}^{s} p_{j} Y^{2} q_{j} - \sum_{j=1}^{s}\Bigl( w_{1}(Y) p_{j} Y q_{j} + p_{j} Y q_{j} w_{2}(Y)\Bigl)$ is a non-trivial GPI for $D$ and by Theorem \ref{m1}, we get $D$ is finite-dimensional over its center.
\end{proof}


\section{The identity $g_{1}(y)y^{-m} + y^{n}g_{2}(y^{-1})=0$}

Although our approach is motivated by \cite{tl1}, the identity \eqref{maineq1} requires a different analysis due to the two independent exponents $m$ and $n$.  In particular, terms involving $g_{1}(1)$ are multiplied by powers of the variable on both sides. We determine all solutions of the equation \eqref{maineq1} in this section.

We notice that the Theorem \ref{main2} may not hold in general. The following example justifies our fact.

\begin{example}
	Let $D = \mathbb{Z}_{5} [Y_{1}, Y_{2}]$. Define additive maps $g_{1},g_{2}$ on $D$ such that $g_{1}|_{\mathbb{Z}_{5}}(y) = 3y^{5}$  and $g_{2}|_{\mathbb{Z}_{5}}(y) = 2y$. For $n=4,m=2$ it is easy to see that $g_1$ and $g_2$ are satisfying the equation \eqref{maineq1} but neither $g_{1} = 0$ nor $g_{2} = 0$.
\end{example}

\begin{prop}\label{t12}
 Let $m$, $n$ be positive integers with $(m,n) \neq (1, 1)$ and $R$ be an algebra over a field $F$. Suppose $g_{1}$ and $g_{2}$ are additive maps on $R$ such that 
	\begin{equation} \label{eq9}
		g_{1}(y)y^{-m} + y^{n}g_{2}(y^{-1})=0
	\end{equation}
	for all $y \in R^*$. Then $g_1=g_2=0$ on $R^*$, unless $\operatorname{char}(F)=p>0$ and $p-1\mid n+m-2.$
\end{prop}
\begin{proof}
 When $\operatorname{char}(F)=0$, we take $k=2$. If $\operatorname{char}(F)=p>0$ and $p-1\nmid n+m-2$, choose $k$ such that $\overline{k}$ is a generator of the cyclic multiplicative group $\mathbb{Z}_p^{*}$.
 Then $k(k^{n+m-2}-1) \neq 0$ in $F$. Substituting $ky$ for $y$ in equation (\ref{eq9}), we obtain
	\begin{align*}
		kg_{1}(y) & = g_{1}(ky)\\
		& = - k^{n} y^{n}g_{2}(k^{-1}y^{-1})k^{m}y^{m}\\
		& = - k^{n+m-1}y^{n}g_{2}(y^{-1})y^{m}\\
		& = k^{m+n-1}g_{1}(y)
	\end{align*}
	for all $y \in R^*$. This implies that $k(k^{n+m-2}-1)g_{1}(y) = 0$ for all $y \in R^*$. Hence $g_{1} = 0$ on $R^*$. It then follows from equation (\ref{eq9}) that $g_2=0$ on $R^*$.
\end{proof}

Specializing Proposition \ref{t12} to the division rings gives the following lemma.
\begin{lem} \label{4t2}
 Suppose $g_{1},g_{2}$ are additive maps on a division ring $D$ and $m,n$ are positive integers with $(m,n) \neq (1,1)$ such that $g_{1}(y)y^{-m} + y^{n}g_{2}(y^{-1})=0$ for all $y \in D^{*}$. Then $g_1=g_2=0$ on $D^*$, unless $\operatorname{char}(D)=p>0$ and $p-1\mid n+m-2$.
\end{lem}

In \cite[Example $4.3$]{tk1}, several examples illustrate that the restriction on the characteristic in Proposition \ref{t12} is necessary.

\begin{rem}
	If either $m=1$ and $n=2$ or $m=2$ and $n=1$, then $g_{1}=g_{2}=0$ by Lemma \ref{4t2}.
\end{rem}

Throughout this section, we  assume that

\begin{equation} \tag{*} \label{1*}
	\operatorname{char} D = p > 2, p-1 | n+m-2, n,m \geq 2
\end{equation}
 and write
$$
n=p^{\ell_1}k_1,\qquad m=p^{\ell_2}k_2,\qquad \ell_i\ge0,\ \ p\nmid k_i,
$$
where $k_i\ge1$ for $i\in\{1,2\}$.
We say $k>1$ is admissible if $k-1$ is not a non-negative power of $p$, and
say $k$ is of binomial type if $k=p^{s}+1$ with $(\ell,s)\neq(0,0)$,
equivalently $p^{\ell}k=p^{\ell+s}+p^{\ell}$. Each $k_i$ is then $1$,
admissible, or of binomial type.

The proof of Theorem \ref{main2} breaks into three cases according to the values of $k_1,k_2$.
\begin{case-x} If we assume $k_1,k_2 >1$, then we divide into the following subcases.
	\begin{subcase} \label{3c1}
		$k_1,k_2$ both admissible.
	\end{subcase}
	\begin{subcase} \label{3c2}
		$k_1$ is admissible, $k_2$ is of binomial type.
	\end{subcase}
	\begin{subcase} \label{3c3}
		$k_1$ is of binomial type and $k_2$ is admissible.
	\end{subcase}
	\begin{subcase} \label{3c4}
		$k_1,k_2$ both of binomial type.
	\end{subcase}
\end{case-x}

\begin{case-x} \label{3c56}
	If we assume exactly one of $k_1$ and $k_2$ equals $1$, then we have the following subcases.
	\begin{subcase} \label{3c5}
		$k_1=1$ and $k_2$ is admissible.
	\end{subcase} \label{3c6}
	\begin{subcase}
		$k_1$ is admissible and $k_2=1$.
	\end{subcase}
\end{case-x}
\begin{case-x} \label{3c7}
	$k_1=k_2=1$.
\end{case-x}

Consider $P_{1}(Y)= 2\Bigl( \sum^{k_{1}}_{i=0} \binom{k_{1}}{i} Y^{p^{\ell_{1}i}} \Bigl), P_{2}(Y) = \sum^{k_{2}}_{j=0} \binom{k_{2}}{j} Y^{p^{\ell_{2}j}}$.

\begin{lem} \label{l7}
 Suppose $k > 1$ is an integer, $p$ is an odd prime satisfying $gcd(p,k)=1$. Assume further that $k-1$ is not a non-negative power of $p$. Let $s \geq 0 $ be the largest integer such that $p^{s} | k-1$. Then the following hold:
	
	\begin{itemize}
		\item[(i)] $p \not| \binom{k}{p^s+1}$,
        \item[(ii)] $p \not| \binom{k}{p^s}$, 
		\item[(iii)] Assume \eqref{1*} and Case $1$. Then
		$P_{1}(Y), P_{2}(Y) \in \mathbb{Z}_p[Y] / \{0\}$.
	\end{itemize}
\end{lem}
	\begin{proof}
		\begin{itemize}
			\item[(i)] Conclusion follows from \cite[Lemma $5.3 (i)$]{tk1}.
                \item[(ii)] Applying similar argument as in \cite[Lemma $5.3 (i)$]{tk1}, we get our conclusion.
			\item[(iii)] As $p$ is odd, $p-1$ is even, and $p-1 | n+m-2$ forces $n+m-2$ is even. Hence $n,m$ have the same parity, and so do $k_{1}, k_{2}$.  Let $s_{1}, s_{2} \geq 0$ be the largest such that $p^{s_{1}} | k_{1}-1$ and $p^{s_{2}}|k_{2}-1$. Since $k_{1}-1$, $k_{2}-1$  are not non-negative power of $p$, we have $p^{s_{1}} < k_{1}-1$ and $p^{s_{2}} < k_{2}-1$. First, we assume that $k_{1}$ is even.  In this case $p^{s_{1}} +1$ is even and $p^{s_{1}}+1 \leq k_{1}-2$.
			If $k_{1}$ is odd, then $p^{s_{1}}$ is odd and $p^{s_{1}} \leq k_{1}-2$. By (i) and (ii), we obtain $p \not| \binom{k_{1}}{p^{s_{1}}+1}$ and $p \not| \binom{k_{1}}{p^{s_{1}}}$. Using similar arguments for $k_{2}$, we obtain $p \not| \binom{k_{2}}{p^{s_{2}}+1}$ and $p \not| \binom{k_{2}}{p^{s_{2}}}$. Thus, $P_{1}(Y)$ and $P_{2}(Y)$ are non-zero in $\mathbb{Z}_p[Y] / \{0\}$.
		\end{itemize}
	\end{proof}

\begin{lem} \label{l8}
	 Suppose $g_{1}$ and $g_{2}$ are additive maps on a non-commutative division ring $D$ and $n,m \geq 2$ are integers such that $g_{1}(y)y^{-m} + y^ng_{2}(y^{-1})=0$, for all $y \in D^{*}$. Then 
	\begin{equation*}
		g_{1}(a_{1}-a_{1}^{2}a_{2}) = (1 - a_{1}a_{2})^ng_{1}(a_{1})(1 -a_{1}a_{2})^{m} - (a_{1}a_{2})^ng_{1}(a_{1})(a_{1}a_{2})^{m}
		- a_{1}^ng_{2}(a_{2})a_{1}^{m}
	\end{equation*}
	for all $ a_{1},a_{2} \in D$, where $a_{1}a_{2} = a_{2}a_{1}$ and $a_1a_2 \neq 1$. In particular 
		\begin{equation*}
		g_{1}(a_{1}) =  g_{1}(1) - (1 - a_{1})^ng_{1}(1)(1 -a_{1})^{m} + a_{1}^ng_{1}(1)a_{1}^{m}
		+ g_{2}(a_{1})
	\end{equation*}
	for all $a_{1} \in D$.
\end{lem}
\begin{proof}
	From Hua's identity, for any $ a_{1},a_{2} \in D^{*}$ such that $a_1 a_{2}\neq 1$ and $a_{1}a_{2} = a_{2}a_{1}$, we have
	\begin{equation*}
		a_{1}-a_{1}^{2}a_{2}= [ a_{1}^{-1} + ( a_{2}^{-1} - a_{1})^{-1} ]^{-1}.
	\end{equation*}
	By the additivity of $g_{1}$, we obtain
	\begin{equation}
		\begin{aligned}
	g_{1}(a_{1}-a_{1}^{2}a_{2})
=&  g_{1}( [a_{1}^{-1} + (a_{2}^{-1} - a_{1})^{-1}]^{-1}) \\
= & -\Bigl((a_{1}^{-1} + (a_{2}^{-1} - a_{1})^{-1})^{-n}g_{2}(a_{1}^{-1} + (a_{2}^{-1} - a_{1})^{-1})\Bigl) \\
&\hspace{0.5cm}\Bigl((a_{1}^{-1} + (a_{2}^{-1} - a_{1})^{-1})^{-m}\Bigl)\\
= & -(a_{1} - a_{1}^{2}a_{2})^{n}\Bigl(-a_{1}^{n}g_{1}(a_{1})a_{1}^{m}\\
&- (a_{2}^{-1} - a_{1})^{-n}g_{1}( a_{2}^{-1} -a_{1} )( a_{2}^{-1} - a_{1} )^{-m} \Bigl)\\ & \hspace{0.4cm} \Bigl(( a_{1} - a_{1}^{2}a_{2} )^{m}\Bigl)\\
= &(a_{1} - a_{1}^{2}a_{2})^{n}a_{1}^{-n}g_{1}(a_{1})a_{1}^{-m}(a_{1} - a_{1}^{2}a_{2})^{m} \\
&+ 
\Bigl((a_{1} - a_{1}^{2}a_{2})^{n}(a_{2}^{-1} - a_{1})^{-n}g_{1}(a_{2}^{-1} - a_{1})\Bigl)\\ & \hspace{0.5cm}\Bigl((a_{2}^{-1} - a_{1})^{-m}(a_{1} - a_{1}^{2}a_{2})^{m}\Bigl)\\
=& (1 - a_{1}a_{2})^ng_{1}(a_{1})(1 -a_{1}a_{2})^{m} - (a_{1}a_{2})^ng_{1}(a_{1})(a_{1}a_{2})^{m}\\
& - (a_{1}a_{2})^na_{2}^{-n}g_{2}(a_{2})a_{2}^{-m}(a_{1}a_{2})^{m}\\
= & (1 - a_{1}a_{2})^ng_{1}(a_{1})(1 -a_{1}a_{2})^{m} - (a_{1}a_{2})^ng_{1}(a_{1})(a_{1}a_{2})^{m}\\
& - a_{1}^ng_{2}(a_{2})a_{1}^{m}.
		\end{aligned}
	\end{equation}
\end{proof}

\begin{lem} \label{l9}
	Suppose $g_{1}$ and $g_{2}$ are additive maps on a non-commutative division ring $D$ satisfying (\ref{1*}) and Subcase \ref{3c1}. If $g_{1}(y)y^{-m} + y^ng_{2}(y^{-1})=0$ for all $y \in D^{*}$, then $g_{1}(1) = 0$ and $g_{1} = g_{2}$.
\end{lem}
\begin{proof}
	Since $p-1$ is even and $p-1\mid n+m-2$, the integer $n+m$ must be even. Hence, both $n$ and $m$ are either even or odd it implies that both $k_{1}$ and $k_{2}$ are either even or odd respectively . Moreover, since $k_{i}-1 \neq p^0=1$, it follows that $k_{i} > 2$, for $i \in \{1,2\}$. By Lemma \ref{l8}, we have 
	\begin{equation} \label{eq1}
		(g_{1} - g_{2})(a_{1})=  g_{1}(1) -  (1-a_{1})^ng_{1}(1)(1-a_{1})^m + a_{1}^n g_{1}(1) a_{1}^m
	\end{equation}
	for all $a_{1} \in D$. Replacing $a_{1}$ by $-a_{1}$ in equation (\ref{eq1}), we get
	\begin{equation} \label{eq2}
		-(g_{1} - g_{2})(a_{1})=  g_{1}(1) -  (1+a_{1})^ng_{1}(1)(1+a_{1})^m + a_{1}^n g_{1}(1) a_{1}^m
	\end{equation}
	for all $a_{1} \in D$. From equation (\ref{eq1}) and equation (\ref{eq2}), we obtain
	\begin{equation} \label{eq3}
		0= 2g_{1}(1) - (1-a_{1})^ng_{1}(1)(1-a_{1})^m - (1+a_{1})^ng_{1}(1)(1+a_{1})^m + 2a_{1}^ng_{1}(1)a_{1}^m
	\end{equation}
	for all $a_{1} \in D$. Assume that $k_{i}$ is even. Since $k_{i} >2$ for $i \in \{1,2\}$, equation (\ref{eq3}) becomes
	\begin{equation} \label{l9eq1}
		\begin{aligned}
		2\Bigl( \sum^{k_{1}-2}_{i=2,even} &\binom{k_{1}}{i} a_{1}^{p^{\ell_{1}}i} \Bigl)g_{1}(1) \Bigl( \sum^{k_{2}}_{j=0, even} \binom{k_{2}}{j} a_{1}^{p^{\ell_{2}}j}\Bigl) \\
		&+ 2 g_{1}(1)\Bigl( \sum^{k_{2}}_{j=2, even} \binom{k_{2}}{j} a_{1}^{p^{\ell_{2}}j}\Bigl)
		 + 2a_{1}^{p^{\ell_{1}}k_{1}}  g_{1}(1)\Bigl( \sum^{k_{2}-2}_{j=0, even} \binom{k_{2}}{j} a_{1}^{p^{\ell_{2}}j}\Bigl)\\
		 &+ 2\Bigl( \sum^{k_{1}-1}_{i=1, odd} \binom{k_{1}}{i} a_{1}^{p^{\ell_{1}}i} \Bigl)g_{1}(1)\Bigl( \sum^{k_{2}-1}_{j=1, odd} \binom{k_{2}}{j} a_{1}^{p^{\ell_{2}}j}\Bigl)=0 
		\end{aligned}
	\end{equation}
	for all $a_{1} \in D$. If the above equation is trivial generalized polynomial identity, then it follows that
	\begin{equation*}
		2 g_{1}(1)\Bigl( \sum^{k_{2}}_{j=2} \binom{k_{2}}{j} a_{1}^{p^{\ell_{2}}j}\Bigl)=0
	\end{equation*}
	for all $a_{1} \in D$.
	By Lemma \ref{l7}, we have $\sum^{k_{2}}_{j=2} \binom{k_{2}}{j} Y^{p^{\ell_{2}}j} \neq 0 \in D_{G}[Y]$. Applying Jacobson's theorem \cite[Theorem 13.11]{tl1}, we obtain an element $a_1\in D$ for which $ \sum^{k_{2}}_{j=2} \binom{k_{2}}{j} a_{1}^{p^{\ell_{2}}j}\neq 0.$
	Thus we get $g_{1}(1) = 0$ and hence $g_{1} = g_{2}$. Similarly, we can prove in case of $k_{i}$ is odd.
	
	On the other hand, we assume that the equation \eqref{l9eq1} is non-trivial GPI. Thus
	\begin{equation*}
		2\Bigl( \sum^{k_{1}}_{i=0} \binom{k_{1}}{i} a_{1}^{p^{\ell_{1}}i} \Bigl)g_{1}(1)\Bigl( \sum^{k_{2}}_{j=0} \binom{k_{2}}{j} a_{1}^{p^{\ell_{2}}j}\Bigl)=0
	\end{equation*}
	for all $a_{1} \in D$, where $(i,j) \neq (0,0), (k_{1},k_{2})$  and both $i$ and $j$ can not be even and odd or odd and even respectively.  This implies that
	\begin{equation*}
		2\Bigl( \sum^{k_{1}}_{i=0} \binom{k_{1}}{i} a_{1}^{p^{\ell_{1}}i} \Bigl)\Bigl( \sum^{k_{2}}_{j=0} \binom{k_{2}}{j} a_{1}^{p^{\ell_{2}}j}\Bigl)g_{1}(1)=0
	\end{equation*}
	for all $a_{1} \in Z(D)$, where $(i,j) \neq (0,0), (k_{1},k_{2})$  and both $i$ and $j$ can not be even and odd or odd and even respectively. Applying Jacobson's theorem \cite[Theorem 13.11]{tl1}, we obtain an element $a_1\in Z(D)$ for which $ \Bigl( \sum^{k_{1}}_{i=0} \binom{k_{1}}{i} a_{1}^{p^{\ell_{1}}i} \Bigl)\Bigl( \sum^{k_{2}}_{j=0} \binom{k_{2}}{j} a_{1}^{p^{\ell_{2}}j}\Bigl)\neq 0.$
	Thus we get $g_{1}(1) = 0$ and hence $g_{1} = g_{2}$.
\end{proof}


\begin{lem} \label{lc9}
	 Suppose $g_{1}$ and $g_{2}$ are additive maps on a non-commutative division ring $D$ satisfying (\ref{1*}) and Subcase \ref{3c2}. If $g_{1}(y)y^{-m} + y^ng_{2}(y^{-1})=0$ for all $y \in D^{*}$, then $g_{1}(1) = 0$ and $g_{1} = g_{2}$.
\end{lem}

\begin{proof}
Since $p$ is an odd prime and $p-1 | n+m -2$, it follows that $n+m$ is even. This implies that both $n$ and $m$ are either even or odd prime and so $k_{1}$ is even. Also, it follows from $k_{1}-1 \neq p^0=1$ that $k_{1} > 2$. By Lemma \ref{l8}, we have
	\begin{equation} \label{eqc1}
		(g_{1}-g_{2})(a_{1})=  g_{1}(1) -  (1-a_{1})^ng_{1}(1)(1-a_{1})^m + a_{1}^n g_{1}(1) a_{1}^m
	\end{equation}
	for all $a_{1} \in D$. Replacing $a_{1}$ by $-a_{1}$ in equation (\ref{eqc1}), we obtain
	\begin{equation} \label{ceq2}
		-(g_{1}-g_{2})(a_{1})=  g_{1}(1) -  (1+a_{1})^ng_{1}(1)(1+a_{1})^m + a_{1}^n g_{1}(1) a_{1}^m
	\end{equation}
	for all $a_{1} \in D$. From equation (\ref{eqc1}) and equation (\ref{ceq2}), we obtain
	\begin{equation} \label{eqc3}
		0= 2g_{1}(1) - (1-a_{1})^ng_{1}(1)(1-a_{1})^m - (1+a_{1})^ng_{1}(1)(1+a_{1})^m + 2a^ng_{1}(1)a_{1}^m
	\end{equation}
	for all $a_{1} \in D$. Since $k_{1} >2$ and $m=p^{\ell_{2}+m_{2}} + p^{\ell_{2}}$, the equation (\ref{eqc3}) reduces to
	\begin{equation*}
		\begin{aligned}
				2\Bigl( & \sum^{k_{1}-2}_{i=2, even} \binom{k_{1}}{i} a_{1}^{p^{\ell_{1}}i} \Bigl)g_{1}(1)(1+a_{1}^{p^{\ell_{2}+m_{2}} + p^{\ell_{2}}}) + 2\Bigl( \sum^{k_{1}-1}_{i=1, odd} \binom{k_{1}}{i} a_{1}^{p^{\ell_{1}}i} \Bigl)g_{1}(1)\\
				&(a_{1}^{p^{\ell_{2}+m_{2}}} + a_{1}^{p^{\ell_{2}}}) + 2a_{1}^{p^{\ell_{1}}k_{1}}g_{1}(1) + 2g_{1}(1)a_{1}^{p^{\ell_{2} + m_{2}} + p^{\ell_{2}}} =0
		\end{aligned}	
	\end{equation*}
 for all $a_{1} \in D$.
	Let $Q_{1}(Y) = 2\Bigl( \sum^{k_{1}-2}_{i=2, even} \binom{k_{1}}{i} Y^{p^{\ell_{1}}i} \Bigl)g_{1}(1)(1+Y^{p^{\ell_{2}+m_{2}} + p^{\ell_{2}}}) + 2\Bigl( \sum^{k_{1}-1}_{i=1, odd} \binom{k_{1}}{i} Y^{p^{\ell_{1}}i} \Bigl)g_{1}(1)(Y^{p^{\ell_{2}+m_{2}}} + Y^{p^{\ell_{2}}}) + 2Y^{p^{\ell_{1}}k_{1}}g_{1}(1) + 2g_{1}(1)Y{p^{\ell_{2} + m_{2}} + p^{\ell_{2}}}.$ 
	
	If $Q_{1}(Y) = 0 \in D_{G}[Y]$, then we get $2Y^{p^{\ell_{1}}k_{1}}g_{1}(1) = 0$, which implies $g_{1}(1) = 0$.
	
Let $Q_{1}(Y) \neq 0 \in D_{G}[Y]$.	This implies that $D$ is a GPI algebra. That is
\begin{equation*}
	\begin{aligned}
		\Bigl( 2\Bigl( & \sum^{k_{1}-2}_{i=2, even} \binom{k_{1}}{i} a_{1}^{p^{\ell_{1}}i} \Bigl)(1+a_{1}^{p^{\ell_{2}+m_{2}} + p^{\ell_{2}}}) +2a_{1}^{p^{\ell_{1}}k_{1}}\\
		&+ 2\Bigl( \sum^{k_{1}-1}_{i=1, odd} \binom{k_{1}}{i} a_{1}^{p^{\ell_{1}}i} \Bigl)(a_{1}^{p^{\ell_{2}+m_{2}}} + a_{1}^{p^{\ell_{2}}})  + 2a_{1}^{p^{\ell_{2} + m_{2}} + p^{\ell_{2}}} \Bigl)g_{1}(1) =0
	\end{aligned}	
\end{equation*}
for all $a_{1} \in Z(D)$.
 According to Jacobson's theorem \cite[Theorem $13.11$]{tl1}, we obtain an element $ a_{1} \in D $ for which $2\Bigl( \sum^{k_{1}-2}_{i=2, even} \binom{k_{1}}{i} a_{1}^{p^{\ell_{1}}i} \Bigl)(1+a_{1}^{p^{\ell_{2}+m_{2}} + p^{\ell_{2}}}) +2a_{1}^{p^{\ell_{1}}k_{1}} + 2\Bigl( \sum^{k_{1}-1}_{i=1, odd} \binom{k_{1}}{i} a_{1}^{p^{\ell_{1}}i} \Bigl)(a_{1}^{p^{\ell_{2}+m_{2}}} + a_{1}^{p^{\ell_{2}}})  + 2a_{1}^{p^{\ell_{2} + m_{2}} + p^{\ell_{2}}}\neq 0$. Thus, $g_{1}(1)=0$ and $g_{1}=g_{2}$.
\end{proof}

Using similar techniques with some necessary variations we can prove the following Lemma. 
\begin{lem} \label{clc9}
	 Suppose $g_{1}$ and $g_{2}$ are additive maps on a non-commutative division ring $D$ satisfying (\ref{1*}) and Subcase \ref{3c3}. If $g_{1}(y)y^{-m} + y^ng_{2}(y^{-1})=0$ for all $y \in D^{*}$, then $g_{1}(1) = 0$ and $g_{1} = g_{2}$.
\end{lem}

\begin{lem} \label{ll6}
	 Suppose $g_{1}$ and $g_{2}$ are additive maps on a non-commutative division ring $D$ satisfying (\ref{1*}) and Subcase \ref{3c4}. If $g_{1}(y)y^{-m} + y^ng_{2}(y^{-1})=0$ for all $y \in D^{*}$, then $g_{1}(1) = 0$ and $g_{1} = g_{2}$.
\end{lem}

\begin{proof}
	From Lemma \ref{l8}, it follows that
	\begin{equation} \label{leq1}
		(g_{1} - g_{2})(a_{1})=  g_{1}(1) -  (1-a_{1})^ng_{1}(1)(1-a_{1})^m + a_{1}^n g_{1}(1) a_{1}^m
	\end{equation}
	for all $a_{1} \in D$. 
	Replacing $a_{1}$ by $-a_{1}$ in equation(\ref{leq1}), we obtain
	\begin{equation} \label{leq2}
		-(g_{1} - g_{2})(a_{1})=  g_{1}(1) -  (1+a_{1})^ng_{1}(1)(1+a_{1})^m + a_{1}^n g_{1}(1) a_{1}^m
	\end{equation}
	for all $a_{1} \in D$.  Since $n= p^{\ell_{1}+m_{1}} + p^{\ell_{1}}$ and $m= p^{\ell_{2}+m_{2}} + p^{\ell_{2}}$, from equation (\ref{leq1}) and equation (\ref{leq2}), we have
	\begin{equation} \label{leq3}
		(g_{1} - g_{2})(a_{1}) = (1+a_{1}^{n})g_{1}(1)(a_{1}^{p^{\ell_{2}}} + a_{1}^{p^{\ell_{2} +m_{2}}}) + (a_{1}^{p^{\ell_{1}}} + a_{1}^{p^{\ell_{1} +m_{1}} })g_1(1)(1+a_{1}^{m})
	\end{equation}
	for all $a_{1} \in D$. 
	If $g_{1}(1) \neq 0$, then the map $a_{1} \mapsto (1+a_{1}^{n})g_{1}(1)(a_{1}^{p^{\ell_{2}}} + a_{1}^{p^{\ell_{2} +m_{2}}}) + (a_{1}^{p^{\ell_{1}}} + a_{1}^{p^{\ell_{1} +m_{1}} })g_1(1)(1+a_{1}^{m})$ is additive. We define
	\begin{equation*}
		\begin{aligned}
		Q(X,Y) = &(1+(X+Y)^{n})g_{1}(1)((X+Y)^{p^{\ell_{2}}} + (X+Y)^{p^{\ell_{2} +m_{2}}}) \\
		&+ ((X+Y)^{p^{\ell_{1}}} + (X+Y)^{p^{\ell_{1} +m_{1}} })g_1(1)(1+(X+Y)^{m})\\
		& - (1+X^{n})g_{1}(1)(X^{p^{\ell_{2}}} + X^{p^{\ell_{2} +m_{2}}}) - (X^{p^{\ell_{1}}} + X^{p^{\ell_{1} +m_{1}} })g_1(1)(1+X^{m})\\
		& - (1+Y^{n})g_{1}(1)(Y^{p^{\ell_{2}}} + Y^{p^{\ell_{2} +m_{2}}}) - (Y^{p^{\ell_{1}}} + Y^{p^{\ell_{1} +m_{1}} })g_1(1)(1+Y^{m})
	\end{aligned}
	\end{equation*}
	as an element of $ D_{G}[X,Y]$.
	Since $Q(X,Y) \neq 0$, it gives that $D$ is a PI- division ring. By Po\v{s}ner's Theorem \cite{pose}, it follows that $D$ is finite-dimensional over its center. Since $D$ is non-commutative, its center $Z(D)$ must be an infinite field. Let $F$ be a maximal subfield of $D$ such that $Z(D) \subset F$. Then by standard result in ring theory, we have $D \otimes_{Z(D)} F \cong M_{r}(F)$, where $r=\sqrt{dim_{Z(D)}D} > 1$.
	Since $D$ and $M_r(F)$ satisfy the same polynomial identities (see \cite[Corollary, p.$64$]{jacob}), the polynomial $Q(X,Y)$ is also an identity of $M_r(F)$. By substituting $X= e = e_{ii}, Y=  f-e = e_{ij}$, where $f= e_{ii} + e_{ij}$, we get
	\begin{equation*}
		2(I + f)g_{1}(1)f+ 2fg_{1}(1)(I+f) - 2(I+e)g_{1}(1)e-2eg_{1}(1)(I+e)=0.
	\end{equation*}
	That is 
	\begin{equation*}
		cf + fc + 2fcf = ec  +ce + 2ece
	\end{equation*}
	where $c = g_{1}(1)$. It follows that 
	\begin{equation}\label{l9eq2}
		ce_{ij} + e_{ij}c + 2e_{ii}ce_{ij} + 2e_{ij}ce_{ii} + 2e_{ij}ce_{ij}=0. 
	\end{equation}
	Let $c=\sum_{s,t=1}^{r}c_{st}e_{st}$, where $c_{st} \in F$. This reduces the equation \eqref{l9eq2} to
	\begin{equation} \label{l9eq3}
		\sum_{s=1}^{r} c_{si}e_{si} + \sum_{s=1}^{r} c_{jt}e_{it} + 2c_{ii}e_{ij} + 2c_{ji}e_{ii} + 2c_{ji}e_{ij}=0.
	\end{equation}
	This implies that $c_{ji}=0$ for all $i \neq j$. Also, we obtain $c_{jj} + 3c_{ii}=0$. If $\operatorname{char}(D) = 3$, then $c_{jj}=0$. Thus, $g_{1}(1)=0$ and $g_{1} = g_{2}$.
	
	On the other hand, we obtain $c_{jj} = -3c_{ii}$. By symmetry, we get $c_{ii}=-3c_{jj}$. Therefore, $8c_{jj}=0$. Since, $\operatorname{char}(D) \neq 2$ we obtain $c_{jj}=0$ for all $j$. Thus, $g_{1}(1)=0$ and $g_{1} = g_{2}$.
\end{proof}

Using the similar techniques as above, we can prove the following lemma.

\begin{lem}\label{lopower}
	Suppose $g_1,g_2$ are additive maps on a non-commutative division ring $D$  satisfying $(*)$, and Case \ref{3c56}. If $g_1(y)y^{-m}+y^{n}g_2(y^{-1})=0$
	for all $y\in D^{*}$, then $g_{1}=g_{2}$.
\end{lem}

\begin{fact} \cite[Lemma 1.5 ]{vfilip} \label{vfilipi}
	Let $F$ be an infinite field and $r\geq2$. If $A_1, \cdots A_k \in M_{r}(F)$  are non-scalar matrices, then there exists an invertible matrix $Q \in M_r(F)$ such that every conjugate $QA_i Q^{-1}$ has no zero entries.
\end{fact}

\begin{lem}\label{lbpower}
	Suppose $g_1,g_2$ are additive maps on a non-commutative division ring $D$  satisfying $(*)$ and Case \ref{3c7}. If $g_1(y)y^{-m}+y^{n}g_2(y^{-1})=0$
	for all $y\in D^{*}$, then $g_{1}=-g_{2}$.
\end{lem}

\begin{proof}
	Substituting $a_2=1$ in  the equation obtained from Lemma \ref{l8}, we get
	\begin{equation*}
		g_{1}(a_{1}-a_{1}^{2})=(1-a_{1})^{n}g_{1}(a_{1})(1-a_{1})^{m}
		-a_{1}^{n}g_{1}(a_{1})a_{1}^{m}-a_{1}^{n}g_{2}(1)a_{1}^{m}.
	\end{equation*}
	This implies that 
	\begin{equation}\label{lbpeq1}
		g_{1}(a_{1}^{2})=a_{1}^{n}g_{1}(a_{1})+ g_{1}(a_{1})a_{1}^{m} + a_{1}^{n}g_{2}(1)a_{1}^{m}
	\end{equation}
	for all $a_{1} \in D$.  By our hypothesis, we have $g_{1}(y) = -y^{n}g_{2}(y^{-1})y^{m}$ and $g_{1}(1) = -g_{2}(1)$. It follows that 
	\begin{equation} \label{lbpeq2}
		g_{2}(a_{1}^{2})=a_{1}^{n}g_{2}(a_{1})+ g_{2}(a_{1})a_{1}^{m} + a_{1}^{n}g_{1}(1)a_{1}^{m}
	\end{equation}
	for all $a_{1} \in D$. From equation \eqref{lbpeq1} and equation \eqref{lbpeq2}, we obtain that
	\begin{equation} \label{lbpeq3}
		h(a_{1}^{2})=a_{1}^{n}h(a_{1})+ h(a_{1})a_{1}^{m} 
	\end{equation}
	for all  $a_1 \in D$, where $h(a_1) = g_{1}(a_1) + g_{2}(a_1)$. 
	Since $a_{1}^{p^{\ell_{1}}}, a_{1}^{p^{\ell_{2}}} >1$, applying Lemma \ref{l310} to equation \eqref{lbpeq3}, we get $D$ is finite-dimensional over $Z(D)$. If $Z(D)$ is finite, then the division ring $D$ is also finite. By Wedderburn's Theorem (see \cite{wed}), $D$ is commutative, a contradiction. It gives that $Z(D)$ is infinite.  
	By Lemma \ref{l8}, we obtain
	\begin{equation} \label{lbpeq4}
		\begin{aligned}
			&g_1(a_{1}^2a_{2}) =   a_{1}^{n+m}g_2(a_{2}) - a_{1}^{n}a_{2}^{n} g_1(a_{1})-g_1(a_{1}) a_{1}^{m}a_{2}^{m}
		\end{aligned}
	\end{equation}
	for all $a_1 \in Z(D), a_{2} \in D$.
	By symmetry obtain similar equation for $g_2$, that is 
	\begin{equation} \label{lbpeq5}
		\begin{aligned}
			&g_2(a_{1}^2a_{2}) =   a_{1}^{n+m}g_1(a_{2}) - a_{1}^{n}a_{2}^{n} g_2(a_{1})-g_2(a_{1}) a_{1}^{m}a_{2}^{m}
		\end{aligned}
	\end{equation}
	for all $a_1 \in Z(D), a_{2} \in D$. From equation \eqref{lbpeq4} and equation \eqref{lbpeq5}, we obtain 
	\begin{equation}\label{lbpeq6}
		h(a_{1}^2a_{2}) =   a_{1}^{n+m}h(a_{2}) - a_{1}^{n}a_{2}^{n} h(a_{1})-h(a_{1}) a_{1}^{m}a_{2}^{m}
	\end{equation}
	for all $a_1 \in Z(D), a_{2} \in D$.
	Suppose $h(a_{1}) \neq 0$. Then the map $$a_{2} \mapsto a_{1}^{n}a_{2}^{n} h(a_{1})-h(a_{1}) a_{1}^{m}a_{2}^{m} $$ is an additive map say $h^{'}$. 
	Let 
	\begin{equation*}
		\begin{aligned}
			F(X,Y) = h^{'}(X+Y) - h^{'}(X) - h^{'}(Y) \in D_{G}[X,Y].
		\end{aligned}
	\end{equation*}
	Then, $F(X,Y) \neq 0$ GPI for $D$. By Theorem \ref{m1} and non-commutativity of $D$, it follows that $D$ is finite-dimensional over its infinite center.

	Since $D$ is a finite-dimensional non-commutative division algebra over $Z(D)$, Jacobson's theorem \cite[Theorem $13.11$]{tl1} implies that the center $Z(D)$ cannot be algebraic over $\mathbb{Z} /p \mathbb{Z}$. Thus, there exists $\alpha \in Z(D)$ such that $a_{1}^{n}\alpha^{n}- a_{1}^{m}\alpha^{m}  \neq 0$. It follows that $F(\alpha, -\alpha) \neq 0$, a contradiction. Hence $h(a_{1}) = 0$. From equation (\ref{lbpeq6}) it follows that
	\begin{equation} \label{lbpeq7}
		h(a_{1}^{2}a_{2}) = a_{1}^{n+m}h(a_{2})
	\end{equation}
	for all $a_{2} \in D, a_{1} \in Z(D)$.
	By  Lemma \ref{tl6}, we obtain that
	\begin{equation} \label{lbpeq8}
		\begin{aligned}
			&\Bigl(2a_{1}^{p^{\ell_{1}}}- (a_{1} + a_{2})^{p^{\ell_{1}}} -(a_{1} - a_{2})^{p^{\ell_{1}}} \Bigl)h(y_{1}) \\
			& + h(y_{1})\Bigl(2a_{1}^{p^{\ell_{2}}}- (a_{1} + a_{2})^{p^{\ell_{2}}} -(a_{1} - a_{2})^{p^{\ell_{2}}}\Bigl)\\
			&+ \Bigl(2a_{2}^{p^{\ell_{1}}}- (a_{1} + a_{2})^{p^{\ell_{1}}} +(a_{1} - a_{2})^{p^{\ell_{1}}}\Bigl)h(y_{2})\\
			& + h(y_{2}) \Bigl(2a_{2}^{p^{\ell_{2}}}- (a_{1} + a_{2})^{p^{\ell_{2}}} +(a_{1} - a_{2})^{p^{\ell_{2}}}\Bigl)=0
		\end{aligned}
	\end{equation}
	for all $a_1, a_2 \in D$. We write 
	\begin{equation*}
		(Y_{1}+Y_{2})^{p^{\ell_{1}}} = \sum_{j=0}^{p^{\ell_{1}}}P_{j}(Y_{1},Y_{2}),
		\quad
		(Y_{1}+Y_{2})^{p^{\ell_{2}}} = \sum_{j=0}^{p^{\ell_{2}}}Q_{j}(Y_{1},Y_{2})
	\end{equation*}
	where $P_{j}(Y_{1},Y_{2}),Q_{j}(Y_{1},Y_{2}) \in \mathbb{Z}_{p}[Y_{1},Y_{2}]$ are homogeneous in $Y_{2}$ of degree $j$. Suppose $\ell= max\{\ell_{1},\ell_{2} \}$. In the case $\ell_{1} \neq \ell_{2}$, without loss of generality we assume that $\ell_{1}<\ell_{2}$. Then equation (\ref{lbpeq4}) becomes
	\begin{equation} \label{lbpeq9} 
		\begin{aligned}
			& \Bigl( \sum_{j=1}^{\frac{p^{\ell_1}-1}{2}} P_{2j}(a_{1},a_{2})\Bigl)h(a_{1}) + h(a_{1})\Bigl( \sum_{j=1}^{\frac{p^{\ell_2}-1}{2}}Q_{2j}(a_{1},a_{2})\Bigl)\\
			&= \Bigl( \sum_{j=1}^{\frac{p^{\ell_1}-1}{2}} P_{2j-1}(a_{1},a_{2})\Bigl)h(a_{2}) + h(a_{2})\Bigl(  \sum_{j=1}^{\frac{p^{\ell_2}-1}{2}}Q_{2j-1}(a_{1},a_{2}) \Bigl)
		\end{aligned}
	\end{equation}
	for all $a_{1},a_{2} \in D$. where $P_{j}(a_{1},a_{2})= 0 $ for $j > p_{\ell_{1}}$. Replacing  $a_2$ by $\alpha^2a_2$ in equation \eqref{lbpeq9} and using equation \eqref{lbpeq7}, we obtain
	\begin{equation} \label{lbpeq10} 
		\begin{aligned}
			& \sum_{j=1}^{\frac{p^{\ell_1}-1}{2}} \alpha^{4j} \Bigl( P_{2j}(a_{1},a_{2})h(a_{1}) + h(a_{1}) Q_{2j}(a_{1},a_{2})\Bigl)\\
			&= \sum_{j=1}^{\frac{p^{\ell_1}-1}{2}} \alpha^{4j-1+n+m}\Bigl( P_{2j-1}(a_{1},a_{2})h(a_{2}) + h(a_{2})Q_{2j-1}(a_{1},a_{2}) \Bigl)
		\end{aligned}
	\end{equation}
	for all $a_1, a_2 \in D$. Applying the standard Vandermonde argument to solve equation (\ref{lbpeq10}), we have
	\begin{equation} \label{lbpeq11}
		(P_{p^{\ell}-1}(a_{1},a_{2})h(a_{1}) + h(a_{1}) Q_{p^{\ell}-1}(a_{1},a_{2})) =0
	\end{equation}
	for all $a_{1},a_{2} \in D$. Since $\ell_{1} \neq \ell_{2}$, from equation \eqref{lbpeq11}, we get
	\begin{equation*} 
		h(a_{1}) Q_{p^{\ell_2}-1}(a_{1},a_{2})=0
	\end{equation*}
	for all $a_{1},a_{2} \in D$. 
	Let $a_{1} \in D$ such that $h(a_{1}) \neq 0$. This implies that $Q_{p^{\ell_2}-1}(a_1,a_2)=0$ for all $a_{2} \in D$ and  $[a_{2}^{p^{\ell_2}},a_{1}]=[a_{2},Q_{p^{\ell_2}-1}(a_{1},a_{2})]=0$ 
	for all $a_{2} \in D$. Thus, we have
	\begin{equation*}
		D = C_{D}(a_{2}^{p^{\ell_2}-1}) \cup \{ a_{1} \in D | h(a_{1}) =0\}
	\end{equation*}
	for all  $a_{2} \in D$. Since $h \neq 0$,  we get $D = C_{D}(a_{2}^{2p^{\ell}-1})$. Hence, $a_2^{p^{\ell_2}}\in Z(D)$ for all $a_2\in D$. It then follows from Kaplansky's theorem \cite[Theorem]{kap1} that $D$ is commutative, contrary to our assumption.
	
	On the other hand, if $\ell_{1} = \ell_{2}$, then $P_{p^{\ell}-1}(a_{1},a_{2})=Q_{p^{\ell}-1}(a_{1},a_{2})$.
	From equation \eqref{lbpeq11}, we obtain
	\begin{equation} \label{lbpeq12}
		P_{p^{\ell}-1}(a_{1},a_{2})h(a_{1}) + h(a_{1}) P_{p^{\ell}-1}(a_{1},a_{2}) =0
	\end{equation}
	for all $a_{1},a_{2} \in D$.  If $P_{p^{\ell}-1}(a_{1},a_{2})=0$, then we are done as above. Also if $h(a_{1}) \in Z(D)$, then we have $P_{p^{\ell}-1}(a_{1},a_{2})=0$ and we are done in this case also. Therefore, assume $P_{p^{\ell}-1}(a_{1},a_{2}) \neq 0$ and $h(a_{1}) \notin Z(D)$. Since $h(a_{1}) \neq 0$, then $a_{1} \notin Z(D)$, otherwise we will have $h(a_{1})=0$. Then $\psi(Y)= P_{p^{\ell}-1}(a_{1},Y)h(a_{1}) + h(a_{1}) P_{p^{\ell}-1}(a_{1},Y)$ is a non-trivial GPI. By Theorem \ref{m1}, $D$ is finite-dimensional over its center. Since $D$ is non-commutative, it follows that $Z(D)$ is infinite.  Let $F$ be a maximal subfield of $D$ such that $Z(D) \subset F$. Then by standard result in ring theory, we have $D \otimes_{Z(D)} F \cong M_{r}(F)$, where $r = \sqrt{\operatorname{dim}_{Z(D)}(D)}>1$. Since $D$ and $M_{r}(F)$ satisfy the same polynomial identities (see \cite[Corollary, p.$64$]{jacob}), the polynomial $\psi(Y)$ is also identity of $M_{r}(F)$. Suppose $Y= \operatorname{diag}(\lambda_1, \lambda_2, \cdots, \lambda_r)$, 
	where $\lambda_i \in F$ and $\lambda_i \neq \lambda_j$ for $i \neq j$.  We have $P_{p^{\ell}-1}(a_{1},Y)= \sum_{i+j= p^{\ell}-1} Y^{i}a_1Y^{j}, a_{1} = \sum_{s,t=1}^{r}(a_{1})_{st}e_{st}, h(a_{1}) = \sum_{s,t=1}^{r}h(a_{1})_{st}e_{st}$. Then $P_{p^{\ell}-1}(a_{1},Y)_{kl} = (a_{1})_{k\ell}\mu_{k\ell}$, where $\mu_{k\ell} = \frac{\lambda_{k}^{p^{\ell}}- \lambda_{\ell}^{p^{\ell}}}{\lambda_{k}- \lambda_{\ell}} \in F$ and $\mu_{kk}=0$. It follows that 
	$$\psi(Y)_{ij} = \sum_{k=1}^{r}P_{p^{\ell}-1}(a_{1},Y)_{ik}h(a_1)_{kj} + \sum_{k=1}^{r}h(a_1)_{ik}P_{p^{\ell}-1}(a_{1},Y)_{kj} \in F_{G}[\lambda_1, \lambda_2, \cdots, \lambda_r]$$
	for all $i,j=1,2, \cdots, r$.  
	We can assume that $a_1 $ and $h(a_1)$ are non-zero and non-central.  
	By Fact \ref{vfilipi}, there exists an invertible
	$Q\in M_r(F)$ such that, upon replacing $a_1,h(a_1)$ by their
	conjugates, we may assume that
	\begin{equation*}\label{lbpeq13}
		(a_1)_{ij}\neq0\quad\text{and}\quad h(a_1)_{ij}\neq0\qquad\text{for all }i,j.
	\end{equation*}
	Since $F$ is infinite, we may choose infinitely many $r$-tuples $(\lambda_1, \cdots, \lambda_r) \in F^{r}$,
    with pairwise distinct entries. Therefore, $\psi(Y)_{ij} = \psi(\lambda_1,\cdots, \lambda_r)_{ij}=0$ for infinite many $\lambda_1, \cdots, \lambda_r$. By Vandermonde argument, each coefficient must vanish.
	From the coefficients of $\lambda_i^a \lambda_j^b$, $\lambda_i^{p^{\ell}-1}$,  and $\lambda_j^{p^{\ell}-1}$ we obtain
	\begin{equation} \label{lbpeq14}
		(a_1)_{ij}(h(a_{1})_{ii} + h(a_{1})_{jj})=0 
	\end{equation}
	and 
	\begin{equation}
		\sum_{k=1}^{r}(a_1)_{ik}h(a_{1})_{kj} = -h(a_{1})_{ii}(a_1)_{ij} \quad     	\sum_{k=1}^{r}h(a_{1})_{ik}(a_1)_{kj} = -(a_1)_{ij}h(a_{1})_{jj}.
	\end{equation}
	If we assume $r>2$, then from equation \eqref{lbpeq14}, we obtain $h(a_{1})_{ii} = - h(a_{1})_{jj}$ for $i \neq j$. Simlarly, we obtain $h(a_{1})_{jj} = - h(a_{1})_{kk}$ and $h(a_{1})_{ii} = - h(a_{1})_{kk}$ for $k \neq i,j$. Thus, we get $h(a_{1})_{ii}=0$ for all $i$, which is a contradiction.
	
	On the other hand, assume that $Y= \lambda_1 e_{11} + \lambda_2 e_{22} + \mu e_{12}$. Then $Y^{t} =  \lambda_1^{t} e_{11} + \lambda_2^{t} e_{22} + \mu \Lambda_t e_{12}$, where $\Lambda_t  =\dfrac{\lambda_1^{t}-\lambda_2^{t}}{\lambda_1-\lambda_2}$. By using the similar arguments as above and comparing the cofficient $\mu$, we obtain $h(a_{1})_{21} (a_{1})_{21}\frac{2\mu_{12}}{\lambda_1 - \lambda_2}=0$, which is a contradiction as $\lambda_1 \neq \lambda_2, (a_{1})_{21}, h(a_1)_{11} \neq 0$. Thus, we get $h(a_1)=0$ for all $a_{1} \in D$. That is $g_{1}(a_1) = -g_{2}(a_{1})$.
\end{proof}

By Lemmas \ref{4t2}, \ref{l9}, \ref{lc9}, \ref{clc9},  \ref{ll6} and \ref{lopower}, we have the following.
\begin{lem} \label{ll56}
	Suppose $g_{1}$ and $g_{2}$ are additive maps on a non-commutative division ring $D$ with $\operatorname{char}(D) \neq 2$. Let $m\geq 2, n \geq 2$ are positive integers and Case $1$ and Case $2$ holds. If $g_{1}(y)y^{-m} + y^ng_{2}(y^{-1})=0$ for all $y \in D^*$, then $g_{1} = g_{2}$.
\end{lem}

Throughout the subsequent discussion, the following hypotheses will remain in force. 

Let $D$ be a non-commutative division ring satisfying $(\ref{1*})$, and let $g$ be an additive self-map of $D$ satisfying $g(y)y^{-m}+y^{n}g(y^{-1})=0$ for every $y\in D^{*}$.

The aim is to prove $g = 0$. Define $C_{D}(a_{2}) = \{ a_{1} \in D | a_{1}a_{2} = a_{2}a_{1} \}$, the centralizer of $a_{2} \in D$.

\begin{lem} \label{l10}
	Suppose $P(Y;a_{1}) = (1+Y)^ng(a_{1})(1+Y)^m + (1-Y)^ng(a_{1})(1-Y)^m - 2Y^ng(a_{1})Y^m -2g(a_{1}) \in D_{G}[Y]$, where $a_{1} \in D$. Then the following hold:
	\begin{itemize}
		\item[(i)] $g(y^2) = (1+y)^{n}g(y)(1+y)^{m} - y^{n}g(y)y^{m} - g(y)$ for all $y \in D$;
		\item[(ii)] $P(a_{1}a_{2};a_{1}) = 0$ for all $a_{1}, a_{2} \in D$ such that $a_{1}a_{2} = a_{2}a_{1}$;
		\item[(iii)] If $d \in D$ and $a_{1} \in C_D(d)^*$, $P(d;a_{1}) \neq 0$, then $g(C_D(d)) = 0$;
		\item[(iv)] If $ P(Z(D);a_{1})\neq 0$, then $g =0$.
	\end{itemize}	
\end{lem}
\begin{proof}
	Let $a_{1}, a_{2} \in D^{*}$ such that $a_{1} a_{2} \neq1$. Then, by Hua's identity,
	$a_{1}- a_{1}a_{2}a_{1} = [a_{1}^{-1} + (a_{2}^{-1} - a_{1})^{-1}]^{-1}$.
	Since $g$ is an additive map, it follows that
	\begin{equation} \label{eq4}
		\begin{aligned}
			g(a_{1}-a_{1}a_{2}a_{1})
			=&  g( [a_{1}^{-1} + (a_{2}^{-1} - a_{1})^{-1}]^{-1}) \\
			= & -\Bigl((a_{1}^{-1} + (a_{2}^{-1} - a_{1})^{-1})^{-n}g(a_{1}^{-1} + (a_{2}^{-1} - a_{1})^{-1})\Bigl) \\
			&\hspace{0.5cm}\Bigl((a_{1}^{-1} + (a_{2}^{-1} - a_{1})^{-1})^{-m})\Bigl)\\
			= & -(a_{1} - a_{1}a_{2}a_{1})^{n}\Bigl(-a_{1}^{n}g(a_{1})a_{1}^{m}\\
		 &- (a_{2}^{-1} - a_{1})^{-n}g( a_{2}^{-1} -a_{1} )( a_{2}^{-1} - a_{1} )^{-m} \Bigl)\\ & \hspace{0.4cm} \Bigl(( a_{1} - a_{1}a_{2}a_{1} )^{m}\Bigl)\\
			= &(a_{1} - a_{1}a_{2}a_{1})^{n}a_{1}^{-n}g(a_{1})a_{1}^{-m}(a_{1} - a_{1}a_{2}a_{1})^{m} \\
			 &+ 
			\Bigl((a_{1} - a_{1}a_{2}a_{1})^{n}(a_{2}^{-1} - a_{1})^{-n}g(a_{2}^{-1} - a_{1})\Bigl)\\ & \hspace{0.5cm}\Bigl((a_{2}^{-1} - a_{1})^{-m}(a_{1} - a_{1}a_{2}a_{1})^{m}\Bigl)
		\end{aligned}
	\end{equation}
	for all $a_{1},a_{2} \in D$ with $a_{1}a_{2} \neq 0,1$. If $a_{1}a_{2} = a_{2}a_{1}$, then equation (\ref{eq4}) becomes
	\begin{equation*}
		\begin{aligned}
			g(a_{1}-a_{1}^{2}a_{2}) &= (1 - a_{1}a_{2})^ng(a_{1})(1 -a_{1}a_{2})^{m} - (a_{1}a_{2})^ng(a_{1})(a_{1}a_{2})^{m}\\
			& - (a_{1}a_{2})^na_{2}^{-n}g(a_{2})a_{2}^{-m}(a_{1}a_{2})^{m}.
		\end{aligned}	
	\end{equation*}
	This yields that
	\begin{equation}
		\begin{aligned}
          g( a_{1} - a_{1}^2a_{2}) = & (1 - a_{1}a_{2})^ng(a_{1})(1 -a_{1}a_{2})^{m} - (a_{1}a_{2})^ng(a_{1})(a_{1}a_{2})^{m}\\
          & - a_{1}^ng(a_{2})a_{1}^{m},
		\end{aligned}
	\end{equation}
	which implies that
	\begin{equation} \label{eq5}
		\begin{aligned}
		g(a_{1}^2a_{2}) = & -(1-a_{1}a_{2})^ng(a_{1})(1-a_{1}a_{2})^{m} + (a_{1}a_{2})^ng(a_{1})(a_{1}a_{2})^{m}\\
		& + a_{1}^ng(a_{2})a_{1}^{m} + g(a_{1})
		\end{aligned}
	\end{equation}
	for all $ a_{1}, a_{2} \in D$ with $a_{1}a_{2} = a_{2}a_{1}$. Since $p-1 | n+m-2$, it follows that $ n+m $ is even. Replacing $a_{1} $ by $-a_{1}$ in equation (\ref{eq5}), we get
	\begin{equation} \label{eq6}
		\begin{aligned}
				g(a_{1}^2a_{2}) = & (1+a_{1}a_{2})^ng(a_{1})(1+a_{1}a_{2})^{m} - (a_{1}a_{2})^ng(a_{1})(a_{1}a_{2})^{m}\\
				& + a_{1}^ng(a_{2})a_{1}^{m} - g(a_{1})
		\end{aligned}	
	\end{equation}
	for all $a_{1}, a_{2} \in D$ with $a_{1}a_{2} = a_{2}a_{1}$. Since $g(1) = 0$, substituting  $a_{2} = 1$ in equation (\ref{eq6}), we get
	\begin{equation} \label{eq7}
		g(a_{1}^2) = (1+a_{1})^ng(a_{1})(1+a_{1})^{m} - a_{1}^ng(a_{1})a_{1}^{m} - g(a_{1}).
	\end{equation}
	This proves (i). Comparing equation (\ref{eq5}) and equation (\ref{eq6}), we get (ii).
	
	Assume $ d \in D$ such that $P(d;a_{1}) \neq 0$ and $a_{1} \in C_D(d)$. Take $a_{2} = a_{1}^{-1}d$. Then $a_{1}a_{2} = a_{2}a_{1}$. By (ii), we have $P(d;a_{1}) = P(a_{1}a_{2};a_{1}) = 0$, which is a contradiction. Thus, we get $g(a_{1})=0$. That is $g(C_D(d)) = 0$. Thus, the conclusion (iv) follows from (iii).
\end{proof}

\begin{lem} \label{lem9}
	If Subcase \ref{3c1} holds, then $g = 0$.
\end{lem}

\begin{proof}
	Suppose $ g \neq 0$. 
	Let 
	\begin{equation*}
		\begin{aligned}
			P(Y;a_{1}) = (1+Y)^{n}g(a_{1})(1+Y)^{m} + (1-Y)^{n}g(a_{1})(1-Y)^{m} & - 2Y^{n}g(a_{1})Y^{m} \\ 
			& -2g(a_{1}) \in D_{G}[Y].
		\end{aligned}
	\end{equation*}
	By Lemma \ref{l10}(iv), we get $P(Z(D);a_{1})=0$.
	

	\textbf{Case $A$:}  Assume that $P(d;a_{1}) \neq 0$ for all $ d \in D / Z(D) $. Then by Lemma \ref{l10}(iii), $g(C_D(d))=0$ for all $d \in D/Z(D)$. In particular, $g(Z(D))=0$ and $g(d)=0$ for all $d \in D / Z(D)$. Therefore, $g=0$.

	\textbf{Case $B$:} Let $P(d;a_{1})=0$ for some $d \in D / Z(D)$. By using the similar arguments as in Lemma \ref{l9}, we can prove that in case $P(Y;a_1)$ is a trivial GPI, $g(a_{1})=0$ for all $a_1 \in D$. On the other hand, assume that $P(Y;a_1)$ is a GPI. By Theorem \ref{m1}, $[D: Z(D)] \leq \infty$. Since $D$ is non-commutative, it implies that $Z(D)$ is infinite. Also, we have
	\begin{equation*}
				\begin{aligned}
						((1+y)^{n+m} + (1-y)^{n+m} - 2y^{n+m} - 2)g(a_{1})=0
					\end{aligned}
			\end{equation*}
		for all $y \in Z(D)$. By \cite[Theorem $16.7$]{tl1}, there exists $z \in Z(D)$ such that $((1+z)^{n+m} + (1-z)^{n+m} - 2z^{n+m} - 2) \neq 0$. This implies that $g(a_{1})=0$ for all $a_1 \in D$.
	
\end{proof}

By using the similar arguments as in Lemma \ref{lem9}, we obtain the following remarks.

\begin{rem} \label{crem}
 If Subase \ref{3c2} and Subcase \ref{3c3} holds, then $g=0$.
 
\end{rem}
\begin{rem} \label{c1rem}
	If Case \ref{3c56} holds, then $g=0$.
\end{rem}

\begin{lem} \label{ll10}
	Assume that Subcase \ref{3c4} holds. If $a_1 \in Z(D)$, then $(a_{1}^{m} + a_{1}^{n} + (a_{1}^{p^{\ell_{1}}} + a_{1}^{p^{\ell_{1}+m_{1}}})(a_{1}^{p^{\ell_{2}}} + a_{1}^{p^{\ell_{2}+m_{2}}}))g(a_{2}) = 0$ for all $a_{2} \in D$ and $g=0$.
	
\end{lem}

\begin{proof}
	From equation (\ref{eq6}), we obtain
	\begin{equation} \label{leq22 }
		\begin{aligned}
			&g(a_{1}^2a_{2}) =   a_{1}^{n+m}g(a_{2}) + (1+a_{1}^{n}a_{2}^{n}) g(a_{1})(a_{1}^{p^{\ell_{2}}}a_{2}^{p^{\ell_{2}}}  + a_{1}^{p^{\ell_{2} + m_{2}}}a_{2}^{p^{\ell_{2} + m_{2}}}) \\
			&+ (a_{1}^{p^{\ell_{1}}}a_{2}^{p^{\ell_{1}}} + a_{1}^{p^{\ell_{1} + m_{1}}}a_{2}^{p^{\ell_{1} + m_{1}}}) g(a_{1})(1+ a_{1}^{m}a_{2}^{m}) 
		\end{aligned}
	\end{equation}
	for all $a_{2} \in D$.
	
	Suppose $g(a_{1}) \neq 0$. Then the map $$a_{2} \mapsto (1+a_{1}^{n}a_{2}^{n}) g(a_{1})(a_{1}^{p^{\ell_{2}}}a_{2}^{p^{\ell_{2}}}  + a_{1}^{p^{\ell_{2} + m_{2}}}a_{2}^{p^{\ell_{2} + m_{2}}}) + (a_{1}^{p^{\ell_{1}}}a_{2}^{p^{\ell_{1}}} + a_{1}^{p^{\ell_{1} + m_{1}}}a_{2}^{p^{\ell_{1} + m_{1}}}) g(a_{1})(1+ a_{1}^{m}a_{2}^{m}) $$ is an additive map say $g^{'}$. 
	Let 
	\begin{equation*}
		\begin{aligned}
		Q(X,Y) = g^{'}(X+Y) - g^{'}(X) - g^{'}(Y) \in D_{G}[X,Y].
	\end{aligned}
	\end{equation*}
	Then, $Q(X,Y) \neq 0$ GPI for $D$. By Theorem \ref{m1} and non-commutativity of $D$, it follows that $D$ is finite-dimensional over its infinite center.

 Since $D$ is a finite-dimensional non-commutative division algebra over $Z(D)$, Jacobson's theorem \cite[Theorem $13.11$]{tl1} implies that the center $Z(D)$ cannot be algebraic over $\mathbb{Z}/ p \mathbb{Z}$. Thus, there exists $\alpha \in Z(D)$ such that $(1+a_{1}^{n}\alpha^{n})(a_{1}^{p^{\ell_{2}}}\alpha^{p^{\ell_{2}}}  + a_{1}^{p^{\ell_{2} + m_{2}}}\alpha^{p^{\ell_{2} + m_{2}}})+ (a_{1}^{p^{\ell_{1}}}\alpha^{p^{\ell_{1}}} + a_{1}^{p^{\ell_{1} + m_{1}}}a_{2}\alpha^{p^{\ell_{1} + m_{1}}})(1+ a_{1}^{m}\alpha^{m}) \neq 0$. It follows that $Q(\alpha, -\alpha) \neq 0$, a contradiction. Hence $g(a_{1}) = 0$. From equation (\ref{leq22 }) it follows that
	\begin{equation} \label{leq23}
		g(a_{1}^{2}a_{2}) = a_{1}^{n+m}g(a_{2})
	\end{equation}
	for all $a_{2} \in D$. Replacing $a_{1}$ by $a_{1}+1$ in equation (\ref{leq23}), we get 
	\begin{equation*}
		\begin{aligned}
			2g(a_{1}a_{2}) =
			&  g(a_{2})a_{1}^{p^{\ell_{2}}} + g(a_{2})a_{1}^{p^{\ell_{2} + m_{2}}} + g(a_{2})a_{1}^{m}\\
			& \; \; + a_{1}^{p^{\ell_{1}}}g(a_{2}) + a_{1}^{p^{\ell_{1}}}g(a_{2})a_{1}^{p^{\ell_{2}}} + a_{1}^{p^{\ell_{1}}}g(a_{2})a_{1}^{p^{\ell_{2} + m_{2}}} + a_{1}^{p^{\ell_{1}}}g(a_{2})a_{1}^{m} \\
			& \; \;  + a_{1}^{p^{\ell_{1} + m_{1}}}g(a_{2}) + a_{1}^{p^{\ell_{1} + m_{1}}}g(a_{2})a_{1}^{p^{\ell_{2}}} + a_{1}^{p^{\ell_{1} + m_{1}}}g(a_{2})a_{1}^{p^{\ell_{2} + m_{2}}} \\
			& \; \; + a_{1}^{p^{\ell_{1} + m_{1}}}g(a_{2})a_{1}^{m} 
			+ a_{1}^{n}g(a_{2}) + a_{1}^{n}g(a_{2})a_{1}^{p^{\ell_{2}}} + a_{1}^{n}g(a_{2})a_{1}^{p^{\ell_{2} + m_{2}}}
		\end{aligned}
	\end{equation*}
	for all $a_{2} \in D$.
	Replacing $a_{1}$ by $-a_{1}$ in the above equation, we get
	\begin{equation*}
		\begin{aligned}
			-2g(a_{1}a_{2}) 
			= & -g(a_{2})a_{1}^{p^{\ell_{2}}} - g(a_{2})a_{1}^{p^{\ell_{2} + m_{2}}} + g(a_{2})a_{1}^{m}\\
			& \; \; - a_{1}^{p^{\ell_{1}}}g(a_{2}) + a_{1}^{p^{\ell_{1}}}g(a_{2})a_{1}^{p^{\ell_{2}}} + a_{1}^{p^{\ell_{1}}}g(a_{2})a_{1}^{p^{\ell_{2} + m_{2}}} - a_{1}^{p^{\ell_{1}}}g(a_{2})a_{1}^{m} \\
			& \; \;  - a_{1}^{p^{\ell_{1} + m_{1}}}g(a_{2}) + a_{1}^{p^{\ell_{1} + m_{1}}}g(a_{2})a_{1}^{p^{\ell_{2}}} + a_{1}^{p^{\ell_{1} + m_{1}}}g(a_{2})a_{1}^{p^{\ell_{2} + m_{2}}} \\
			& \; \; - a_{1}^{p^{\ell_{1} + m_{1}}}g(a_{2})a_{1}^{m} 
			+ a_{1}^{n}g(a_{2}) - a_{1}^{n}g(a_{2})a_{1}^{p^{\ell_{2}}} - a_{1}^{n}g(a_{2})a_{1}^{p^{\ell_{2} + m_{2}}}
		\end{aligned}
	\end{equation*}
	for all $a_{2} \in D$.
	By adding above two expressions, we obtain
	\begin{equation} \label{mleq24}
		\begin{aligned}
		(a_{1}^{m} + a_{1}^{n} + (a_{1}^{p^{\ell_{1}}} + a_{1}^{p^{\ell_{1}+m_{1}}})(a_{1}^{p^{\ell_{2}}} + a_{1}^{p^{\ell_{2}+m_{2}}}))g(a_{2}) = 0
		\end{aligned}
	\end{equation}
	for all $a_{2} \in  D$. Since $Z(D)$ is not algebraic over $\mathbb{Z} /p \mathbb{Z}$. Thus, there exist $z \in Z(D)$ such that $z^{m} + z^{n} + (z^{p^{\ell_{1}}} + z^{p^{\ell_{1}+m_{1}}})(z^{p^{\ell_{2}}} + z^{p^{\ell_{2}+m_{2}}}) \neq 0$. It follows that $g(a_{2})=0$ for all $a_{2} \in D$.
	
\end{proof}

Now, we deal with the case when $m_{i} = 0$ for $i \in \{1,2\}$.

\begin{lem} \label{l11}
	If $n = 2p^{\ell_{1}}, m = 2p^{\ell_{2}}$ for some integers $\ell_{1},\ell_{2} > 0$, then $g = 0$.
\end{lem}

\begin{proof}
	By Lemma \ref{l10}, we have
	\begin{equation} \label{leq24}
		g(a_{1}^{2}) = (1+a_{1}^{p^{\ell_{1}}})^{2}g(a_{1})(1+a_{1}^{p^{\ell_{2}}})^{2} - a_{1}^{2p^{\ell_{1}}}g(a_{1})a_{1}^{2p^{\ell_{2}}} -g(a_{1})
	\end{equation}
	for all $a_{1} \in D$.
	Replacing $a_{1}$ by $-a_{1}$ in equation \eqref{leq24}, we obtain
	\begin{equation} \label{nleq25}
		g(a_{1}^{2}) = (1-a_{1}^{p^{\ell_{1}}})^{2}g(a_{1})(1-a_{1}^{p^{\ell_{2}}})^{2} + a_{1}^{2p^{\ell_{1}}}g(a_{1})a_{1}^{2p^{\ell_{2}}} +g(a_{1})
	\end{equation}
	for all $ a_{1} \in D$. From equation \eqref{leq24} and equation \eqref{nleq25}, we get
	\begin{equation} \label{nleq26}
		g(a_{1}^{2}) = a_{1}^{2p^{\ell_{1}}}g(a_{1}) + g(a_{1})a_{1}^{2p^{\ell_{2}}} + 4a_{1}^{p^{\ell_{1}}}g(a_{1})a_{1}^{p^{\ell_{2}}}
	\end{equation}
	for all $a_{1} \in D$. Since $a_{1}^{p^{\ell_{1}}}, a_{1}^{p^{\ell_{2}}} >1$, applying Theorem \ref{tcl5} to equation \eqref{nleq26}, we get $D$ is finite-dimensional over $Z(D)$. If $Z(D)$ is finite, then the division ring $D$ is also finite. By Wedderburn's Theorem (see \cite{wed}), $D$ is commutative, a contradiction. It gives that $Z(D)$ is infinite. 
	Also, by Lemma \ref{ll10}, we have 
	\begin{equation} \label{leq25}
		2g(\alpha a_{2}) = \beta g(a_{2})
	\end{equation}
	where $\beta \in Z(D)$, for all $\alpha \in Z(D)$ and $a_{2} \in D$.
	
	Note that $2^{p^{\ell_{i}}} \equiv 2 \ mod \ p$ for $ i \in \{1,2\}$. Applying Lemma \ref{tl6} to equation (\ref{nleq26}), we obtain
	\begin{equation} \label{leq26}
		\begin{aligned}
			&\Bigl(6a_{1}^{2p^{\ell_{1}}} -(a_{1}+a_{2})^{2p^{\ell_{1}}} -(a_{1}-a_{2})^{2p^{\ell_{1}}} \Bigl)g(a_{1}) + g(a_{1})\Bigl(6a_{1}^{2p^{\ell_{2}}}\\
			& -(a_{1}+a_{2})^{2p^{\ell_{2}}} -(a_{1}-a_{2})^{2p^{\ell_{2}}}\Bigl) + 24a_{1}^{p^{\ell_{1}}}g(a_{1})a_{1}^{p^{\ell_{2}}}\\ &-4(a_{1}+ a_{2})^{p^{\ell_{1}}}g(a_{1})(a_{1} + a_{2})^{p^{\ell_{2}}} -4(a_{1} - a_{2})^{p^{\ell_{1}}}g(a_{1})(a_{1} -a_{2})^{p^{\ell_{2}}} \\= 
			& \Bigl(-2a_{2}^{2p^{\ell_{1}}} +(a_{1}+a_{2})^{2p^{\ell_{1}}} - (a_{1}-a_{2})^{2p^{\ell_{1}}}\Bigl)g(a_{2}) \\
			&+ g(a_{2}) \Bigl(-2a_{2}^{2p^{\ell_{1}}} +(a_{1}+a_{2})^{2p^{\ell_{1}}} - (a_{1}-a_{2})^{2p^{\ell_{2}}}\Bigl)\\
			& +4(a_{1} + a_{2})^{p^{\ell_{1}}}g(a_{2})(a_{1} + a_{2})^{p^{\ell_{2}}} -4(a_{1}- a_{2})^{p^{\ell_{1}}}g(a_{2})(a_{1} -a_{2})^{p^{\ell_{2}}}\\
			& -8a_{2}^{p^{\ell_{1}}}g(a_{2})a_{2}^{p^{\ell_{2}}}
		\end{aligned}
	\end{equation}
	for all $a_{1},a_{2} \in D$. We write 
	\begin{equation*}
		(Y_{1}+Y_{2})^{p^{\ell_{1}}} = \sum_{j=0}^{p^{\ell_{1}}}P_{j}(Y_{1},Y_{2}),
\quad
		(Y_{1}+Y_{2})^{2p^{\ell_{1}}} = \sum_{j=0}^{2p^{\ell_{1}}}U_{j}(Y_{1},Y_{2})
	\end{equation*}
	and
 	\begin{equation*}
		(Y_{1}+Y_{2})^{2p^{\ell_{2}}} = \sum_{j=0}^{2p^{\ell_{2}}}V_{j}(Y_{1},Y_{2}),
\quad
		(Y_{1}+Y_{2})^{p^{\ell_{2}}} = \sum_{j=0}^{p^{\ell_{2}}}Q_{j}(Y_{1},Y_{2})
	\end{equation*}
	where $P_{j}(Y_{1},Y_{2}),Q_{j}(Y_{1},Y_{2}), U_{j}(Y_{1},Y_{2}), V_{j}(Y_{1},Y_{2})  \in \mathbb{Z}_{p}[Y_{1},Y_{2}]$ are homogeneous in $Y_{2}$ of degree $j$. Suppose $\ell= max\{\ell_{1},\ell_{2} \}$. In the case $\ell_{1} \neq \ell_{2}$, without loss of generality we assume that $\ell_{1}<\ell_{2}$. Then equation (\ref{leq26}) becomes
		\begin{align} \label{leq27} 
			& \Bigl( 4 a_{1}^{2p^{\ell_{1}}} - \sum_{j=1}^{p^{\ell}} U_{2j}(a_{1},a_{2})\Bigl)g(a_{1}) + g(a_{1})\Bigl( 4a_{1}^{2p^{\ell_{2}}} - \sum_{j=1}^{p^{\ell}}V_{2j}(a_{1},a_{2})\Bigl)\nn \\
			&+ 24a_{1}^{2p^{\ell_{1}}}g_{1}(a_{1})a_{1}^{2p^{\ell_{2}}}+ \sum_{j=0,k=0}^{\frac{p^{\ell}-1}{2},\frac{p^{l}-1}{2}}
			4P_{2j}(a_{1},a_{2})g(a_{1})Q_{2k}(a_{1},a_{2}) \nn \\
			& + \sum_{j=1,k=1}^{\frac{p^{\ell}-1}{2},\frac{p^{\ell}-1}{2}}
			4P_{2j-1}(a_{1},a_{2})g(a_{1})Q_{2k-1}(a_{1},a_{2}) \nn \\
			&= \Bigl( \sum_{j=1}^{p^{\ell}} U_{2j-1}(a_{1},a_{2}) -2 a_{2}^{2p^{\ell_{1}}}\Bigl)g(a_{1}) + g(a_{2})\Bigl(  \sum_{j=1}^{p^{\ell}}V_{2j-1}(a_{1},a_{2}) -2 a_{2}^{2p^{\ell_{1}}} \Bigl) \\
			& + \sum_{j=0,k=1}^{\frac{p^{\ell}-1}{2},\frac{p^{\ell}-1}{2}}
			4P_{2j}(a_{1},a_{2})g(a_{2})Q_{2k-1}(a_{1},a_{2}) -8a_{2}^{p^{\ell_{1}}}g(a_{2})a_{2}^{p^{\ell_{2}}}\nn \\
			& + \sum_{j=1,k=0}^{\frac{p^{\ell}-1}{2},\frac{p^{\ell}-1}{2}}
		   	4P_{2j-1}(a_{1},a_{2})g(a_{2})Q_{2k }(a_{1},a_{2}) \nn
		\end{align}
	for all $a_{1},a_{2} \in D$, where $P_{j}(a_{1},a_{2})= 0 $ for $j > p^{\ell_{1}}$ and $U_{j}(a_{1},a_{2}) = 0$ for $j> 2p^{\ell_{1}}$. Next, by replacing $a_2$ with $\alpha a_2$ in equation \eqref{leq27} and then multiplying both sides by $2$, we obtain
	
		\begin{align} \label{leq28}
			& \Bigl(4a_{1}^{2p^{\ell_{1}}} - \sum_{j=1}^{p^{\ell}} 2 \alpha^{2j}U_{2j}(a_{1},a_{2})\Bigl)g(a_{1}) + g(a_{1})\Bigl(4a_{1}^{2p^{\ell_{1}}}  -\sum_{j=1}^{p^{\ell}} 2 \alpha^{2j}V_{2j}(a_{1},a_{2})\Bigl) \nn \\
			&+ 24a_{1}^{2p^{\ell_{1}}}g_{1}(a_{1})a_{1}^{2p^{\ell_{2}}}+ \sum_{j=0,k=0}^{\frac{p^{\ell}-1}{2},\frac{p^{\ell}-1}{2}}
			8 \alpha^{2j + 2k} P_{2j}(a_{1},a_{2})g(a_{1})Q_{2k}(a_{1},a_{2}) \nn \\
			& + \sum_{j=1,k=1}^{\frac{p^{\ell}-1}{2},\frac{p^{\ell}-1}{2}}
			8 \alpha^{2j-1 + 2k -1} P_{2j-1}(a_{1},a_{2})g(a_{1})Q_{2k-1}(a_{1},a_{2}) \nn
			\\
			&= \beta \Bigl(\sum_{j=1}^{p^{\ell}} \alpha^{2j-1} U_{2j-1}(a_{1},a_{2})-2 \alpha^{2p^{\ell_{1}}}a_{2}^{2p^{\ell_{1}}}\Bigl)g(a_{2}) \\
			& + \beta g(a_{2})\Bigl(\sum_{j=1}^{p^{\ell}} \alpha^{2j-1}V_{2j-1}(a_{1},a_{2})-2 \alpha^{2p^{\ell_{1}}}a_{2}^{2p^{\ell_{1}}}\Bigl)-8\beta \alpha^{2p^{\ell_{2}}} a_{2}^{p^{\ell_{1}}}g(a_{2})a_{2}^{p^{\ell_{2}}} \nn \\
			& + \sum_{j=0,k=1}^{\frac{p^{\ell}-1}{2},\frac{p^{\ell}-1}{2}}
			4\alpha^{2j +2k -1} \beta P_{2j}(a_{1},a_{2})g(a_{2})Q_{2k-1}(a_{1},a_{2}) \nn \\ 
			& + \sum_{j=1,k=0}^{\frac{p^{\ell}-1}{2},\frac{p^{\ell}-1}{2}}
			4\alpha^{2j-1 + 2k} \beta P_{2j-1}(a_{1},a_{2})g(a_{2})Q_{2k }(a_{1},a_{2})\nn
		\end{align}
	for all $a_{1},a_{2} \in D$ and $\alpha \in Z(D)$. Since $Z(D)$ is infinite. Applying the standard Vandermonde argument to solve equation (\ref{leq28}), we have
	\begin{equation} \label{eq30}
	 2(U_{2p^{\ell}}(a_{1},a_{2})g(a_{1}) + g(a_{1}) V_{2p^{\ell}}(a_{1},a_{2})) =0
	\end{equation}
	for all $a_{1},a_{2} \in D$ and $\alpha \in Z(D)$. Since $\ell_{1} \neq \ell_{2}$, from equation \eqref{eq30}, we get
	\begin{equation} \label{nleq31}
		g(a_{1}) V_{2p^{\ell}_2}(a_{1},a_{2})=0
	\end{equation}
	for all $a_{1},a_{2} \in D$. 
	Let $a_{1} \in D$ such that $g(a_{1}) \neq 0$. This implies that $V_{2p^{\ell_2}}(a_1,a_2)=0$ for all $a_{2} \in D$.  That is $a^{2p^{\ell_{2}}} = 0$, a contradiction. Thus, $g(a_{1})=0$ for all $a_1 \in D$.
	
	On the other hand, if $\ell_{1} = \ell_{2}$, then $U_{2p^{\ell}}(a_{1}, a_{2}) = U_{2p^{\ell}}(a_{1}, a_{2})$, and equation \eqref{eq30} becomes
	\begin{equation*}
		U_{2p^{\ell}}(a_{1},a_{2})g(a_{1}) + g(a_{1}) U_{2p^{\ell}}(a_{1},a_{2})=0
	\end{equation*}
	for all $a_1,a_2 \in D$. This implies that $2a^{2p^{\ell}}g(a_1)=0$ for all $a_1 \in D, a_2 \in Z(D)$. Since $\operatorname{char}(D) \neq 2$ and $D$ is a divsion ring, it follows that $g(a_{1})=0$ for all $a_1 \in D$.
\end{proof}

Now, using the similar arguments used in the proof of Lemma \ref{lbpower} and Lemma \ref{l11}, following lemma can be proved.
\begin{lem} \label{lbp2}
	Suppose $g$ is an additive maps on a noncommutative division ring $D$  satisfying $(*)$ and Case $3$. If $g(y)y^{-m}-y^{n}g(y^{-1})=0$
	for all $y\in D^{*}$, then $g=0$.
\end{lem}

\begin{proof}[ \textbf{Proof of Theorem \ref{main2}}]
	Lemma \ref{4t2} reduces us to
	\begin{equation*}
		\operatorname{char}(D) = p >2, n,m>2, p-1 | n+m-2.
	\end{equation*}
	By Lemma \ref{ll56}, we may assume that $ g_{2} = g_{1}$. That is
	\begin{equation*}
		g_{1}(y)y^{-m} +y^ng_{1}(y^{-1})=0
	\end{equation*}
	holds for all $y \in D^{*}$. From Remarks \ref{crem}, \ref{c1rem} and Lemmas \ref{lem9}, \ref{ll10}, \ref{l11}, \ref{lbp2}, we obtain $g_{1} =g_{2} = 0$.

\end{proof}
%

\section{Acknowledgment}
The author sincerely thanks the reviewers and referees for their helpful remarks and valuable suggestions, which led to a improvement of the paper. First author is supported by DST-INSPIRE (IF$230146$), Ref. No. DST/INSPIRE/03/2023/002370. All authors are equally contributed.

\end{document}